\documentclass{amsart}
\usepackage{fullpage}
\usepackage{amsmath,amssymb,amsthm,mathtools}
\usepackage{microtype}
\usepackage{hyperref}
\newtheorem{theorem}{Theorem}[section]
\newtheorem{proposition}[theorem]{Proposition}
\newtheorem{lemma}[theorem]{Lemma}
\theoremstyle{remark}
\newtheorem{remark}[theorem]{Remark}

\title{Uniform Resolvent Estimates for the Discrete Schr\"odinger Operator in Higher Dimensions}

\author{Yuda Chen}
\email{\href{mailto:math.yuda.chen@gmail.com}{math.yuda.chen@gmail.com}}
\address{School of Mathematical Sciences and LPMC, Nankai University, Tianjin 300071, China}
\date{}
\begin{document}

\begin{abstract}
Let \(H_0\) be the standard discrete Laplacian on \(\mathbb Z^d\),
\(d\geq4\), let \(R_0(z)=(H_0-z)^{-1}\), and let \(p'\) denote the
H\"older conjugate of \(p\), with \(p'=\infty\) when \(p=1\).
We establish uniform diagonal resolvent estimates from
\(\ell^p(\mathbb Z^d)\) to \(\ell^{p'}(\mathbb Z^d)\) by proving local
Fourier-decay bounds for surface measure on the Fermi surfaces of the
lattice dispersion relation.  At a regular point, the decay is governed
by the number of coordinate directions in which the quadratic term
vanishes: after a local change of variables, these directions produce
cubic one-dimensional phases, while the remaining directions are
quadratic.  For \(d\geq5\), we prove the sharp global range
\[
1\leq p\leq\frac{2(d+2)}{d+5}.
\]
Away from the threshold energies, this range remains sharp in odd
dimensions, whereas for even \(d\geq6\) the sharp range improves to
\[
1\leq p\leq\frac{2(2d+5)}{2d+11}.
\]
In dimension four, we obtain uniform estimates for \(1\leq p<4/3\) and
an \(\ell^{4/3}\)-to-\(\ell^4\) endpoint estimate with a square-root
logarithmic loss.  Matching anisotropic Knapp-type constructions give the
corresponding necessary conditions.  We also derive Birman--Schwinger and
Kato smoothing bounds, complex-eigenvalue estimates, and improved thin
spectral projection estimates.
\end{abstract}

\maketitle
\section{Introduction}
Uniform resolvent estimates are a discrete analogue of uniform Sobolev
inequalities and form a basic input in the spectral and scattering theory
of Schr\"odinger operators.  For the continuous Laplacian on
\(\mathbb R^d\), \(d\geq3\), the scale-invariant estimate
\[
\sup_{z\in\mathbb C\setminus[0,\infty)}
\|(-\Delta-z)^{-1}\|_{
\mathcal B(L^{2d/(d+2)}(\mathbb R^d),L^{2d/(d-2)}(\mathbb R^d))}
<\infty
\]
is part of the classical uniform Sobolev theory of Kenig, Ruiz, and
Sogge \cite{KenigRuizSogge1987}.  The lattice problem has a substantially
different geometry.  The frequency space is compact, there is no useful
dilation symmetry, and the level sets of the dispersion relation change
type with the energy.  In particular, they may have vanishing curvature
at regular points and become singular at threshold energies.

Weighted resolvent and Birman--Schwinger estimates for the discrete
Laplacian were developed by Korotyaev and M{\o}ller
\cite{KorotyaevMoller2019} and by Tadano and Taira
\cite{TadanoTaira2019}.  In particular, writing \(H_0\) for the standard
discrete Laplacian and \(p'\) for the H\"older conjugate of \(p\), with
\(p'=\infty\) when \(p=1\), Tadano and Taira proved the
\[
\sup_{z\in\mathbb C\setminus[0,4d]}
\|(H_0-z)^{-1}\|_{\mathcal B(\ell^p(\mathbb Z^d),\ell^{p'}(\mathbb Z^d))}
<\infty
\]
for \(d\geq4\) and \(1\leq p\leq 2d/(d+3)\), and also showed that the
continuous Sobolev exponent \(2d/(d+2)\) is not available when
\(d\geq5\).  The Fourier transform of surface-carried measures associated
with the cubic lattice had earlier been studied by Erd\H{o}s and
Salmhofer \cite{ErdosSalmhofer2007}.  In dimension three, Taira
\cite{Taira2021} used local Fourier decay on the Fermi surfaces to obtain
the sharp away-from-threshold range \(1\leq p\leq5/4\), together with
the global range \(1\leq p\leq6/5\).

For \(d\geq5\), the present paper determines the optimal global
diagonal range and the optimal away-from-threshold diagonal range.  The
global endpoint is
\[
\frac{2(d+2)}{d+5},
\]
which strictly improves \(2d/(d+3)\) and is sharp.  Away from the
threshold energies, the same endpoint remains sharp in odd dimensions,
whereas in even dimensions \(d\geq6\) it improves sharply to
\[
\frac{2(2d+5)}{2d+11}.
\]
Dimension four is borderline: we prove uniform estimates for
\(1\leq p<4/3\) and an endpoint estimate with a square-root logarithmic
loss.  The necessary construction shows that \(p\leq4/3\) is required
globally, but no claim is made here that the logarithmic loss is
necessary; see Remark~\ref{rem:d4-log}.  The resulting resolvent and
Fourier-decay estimates also yield Birman--Schwinger, smoothing,
complex-eigenvalue, and thin spectral projection bounds; see
Section~\ref{sec:applications}.

The higher-dimensional problem is not a formal iteration of the
three-dimensional argument.  A regular point may have any number of
coordinate directions in which the quadratic term vanishes.  After a
suitable local change of variables, these directions produce cubic
one-dimensional phases, while the remaining directions are quadratic.
The number of cubic directions controls both the Fourier decay and the
resolvent range.  The worst regular points occur at the central energy
\(2d\); whether this energy is itself a threshold is responsible for the
parity distinction in the away-from-threshold part of
Theorem~\ref{thm:resolvent}.

\subsection*{Main results}
Let \(d\geq4\) and let
\(\mathbb T^d=(-1/2,1/2]^d\).
For \(f:\mathbb Z^d\to\mathbb C\), we use the Fourier transform
\[
\widehat f(\xi)
=
\sum_{n\in\mathbb Z^d}f(n)e^{-2\pi i n\cdot\xi},
\qquad
f(n)
=
\int_{\mathbb T^d}e^{2\pi i n\cdot\xi}\widehat f(\xi)\,d\xi.
\]
For a bounded function \(m\) on \(\mathbb T^d\), we write \(m(D)\) for
the Fourier multiplier with symbol \(m\).
The discrete Laplacian \(H_0\) is the Fourier multiplier with symbol
\[
h_0(\xi)
=
4\sum_{j=1}^d\sin^2(\pi\xi_j)
=
2d-2\sum_{j=1}^d\cos(2\pi\xi_j),
\]
and
\[
R_0(z)=(H_0-z)^{-1},
\qquad
z\in\mathbb C\setminus[0,4d].
\]
We first state the uniform limiting estimates for
\(z\in\mathbb C\setminus\mathbb R\).  Since the multiplier
\((h_0-z)^{-1}\) is smooth for real \(z\notin[0,4d]\), these estimates
extend to the full resolvent set \(\mathbb C\setminus[0,4d]\).  The
critical points and critical values of \(h_0\) are
\[
\operatorname{Cr}(h_0)
=
\left\{\xi\in\mathbb T^d:\xi_j\in\{0,1/2\}\right\},
\qquad
\{0,4,\dots,4d\},
\]
respectively. For \(0\leq\lambda\leq4d\), set
\[
M_\lambda=\{\xi\in\mathbb T^d:h_0(\xi)=\lambda\}.
\]
A point \(\xi\in M_\lambda\) is called regular if
\(\nabla h_0(\xi)\neq0\), and the regular part of \(M_\lambda\) is the
set of all its regular points.  On the regular part, \(M_\lambda\) is a
smooth \((d-1)\)-dimensional hypersurface, and
\(d\sigma_{M_\lambda}\) denotes its induced Euclidean surface measure.
All integrals against \(d\sigma_{M_\lambda}\) are understood over the
regular part of \(M_\lambda\).  For a smooth cutoff \(\chi\), write
\[
\widehat{\chi\,d\sigma_{M_\lambda}}(x)
=
\int_{M_\lambda}
e^{-2\pi i x\cdot\xi}\chi(\xi)\,d\sigma_{M_\lambda}(\xi),
\qquad x\in\mathbb Z^d.
\]

\begin{theorem}[Local Fourier decay]\label{thm:local-decay}
Let \(\xi^0\in M_\lambda\), and let \(\chi\) be supported sufficiently
close to \(\xi^0\).

Suppose first that \(\xi^0\) is a regular point, and put
\(m=\#\{j:\cos(2\pi\xi_j^0)=0\}\).
If \(0\leq m\leq4\), then
\[
\left|\widehat{\chi\,d\sigma_{M_\lambda}}(x)\right|
\leq
C(1+|x|)^{-(d-2)/2}
\bigl(\log(2+|x|)\bigr)^{m/4}.
\]
If \(m\geq5\), then
\[
\left|\widehat{\chi\,d\sigma_{M_\lambda}}(x)\right|
\leq
C(1+|x|)^{-(d-m)/2-(m-1)/3}.
\]

If \(\xi^0\in\operatorname{Cr}(h_0)\), then
\[
\left|\widehat{\chi\,d\sigma_{M_\lambda}}(x)\right|
\leq
C(1+|x|)^{-(d-2)/2}.
\]
At \(\lambda=0\) and \(\lambda=4d\), there is no local
\((d-1)\)-dimensional regular part near the corresponding critical point.
\end{theorem}

\begin{theorem}[Uniform resolvent estimates]\label{thm:resolvent}
Let \(p'\) be the H\"older conjugate of \(p\), with the convention
\(p'=\infty\) when \(p=1\).  For a function
\(W:\mathbb Z^d\to\mathbb C\), we use the same letter
\(W\) for the corresponding multiplication operator, defined by
\((Wf)(n)=W(n)f(n)\).  Here and below, the corresponding weighted
estimate means that, for all \(W_1,W_2\in\ell^r(\mathbb Z^d)\),
\[
\sup_z
\|W_1R_0(z)W_2\|_{\mathcal B(\ell^2(\mathbb Z^d))}
\leq
C\|W_1\|_{\ell^r(\mathbb Z^d)}
\|W_2\|_{\ell^r(\mathbb Z^d)},
\]
where the supremum is taken over the same set of spectral parameters as
in the corresponding unweighted estimate.

If \(d=4\), then
\(\sup_{z\in\mathbb C\setminus\mathbb R} \|R_0(z)\|_{\mathcal B(\ell^p(\mathbb Z^4),\ell^{p'}(\mathbb Z^4))} <\infty\)
for \(1\leq p<4/3\).  Equivalently, the corresponding weighted estimate
holds for \(2\leq r<4\); it also holds for \(1\leq r<2\) by the inclusion
\(\ell^r\subset\ell^2\).  At the endpoint,
\[
\|R_0(z)\|_{\mathcal B(\ell^{4/3}(\mathbb Z^4),\ell^4(\mathbb Z^4))}
\leq
C\left(
\log\left(2+\frac1{\operatorname{dist}(z,[0,16])}\right)
\right)^{1/2},
\]
and
\[
\|W_1R_0(z)W_2\|_{\mathcal B(\ell^2(\mathbb Z^4))}
\leq
C\left(
\log\left(2+\frac1{\operatorname{dist}(z,[0,16])}\right)
\right)^{1/2}
\|W_1\|_{\ell^4}\|W_2\|_{\ell^4}.
\]

If \(d\geq5\), then
\[
\sup_{z\in\mathbb C\setminus\mathbb R}
\|R_0(z)\|_{\mathcal B(\ell^p(\mathbb Z^d),\ell^{p'}(\mathbb Z^d))}
<\infty
\]
for
\[
1\leq p\leq\frac{2(d+2)}{d+5},
\]
and the corresponding weighted estimate holds for
\[
1\leq r\leq\frac{2(d+2)}3.
\]

Let \(\delta>0\) and set
\(D_\delta = \left\{z\in\mathbb C: \min_{0\leq j\leq d}|z-4j|\geq\delta\right\}\).
If \(d\geq5\) is odd, the preceding ranges remain valid when
the supremum is restricted to \(D_\delta\setminus\mathbb R\).  If
\(d\geq6\) is even, then
\(\sup_{z\in D_\delta\setminus\mathbb R} \|R_0(z)\|_{\mathcal B(\ell^p(\mathbb Z^d),\ell^{p'}(\mathbb Z^d))} <\infty\)
for
\(1\leq p\leq\frac{2(2d+5)}{2d+11}\),
and the corresponding weighted estimate holds for
\(1\leq r\leq\frac{2d+5}{3}\).
\end{theorem}


\begin{theorem}[Necessary conditions]\label{thm:sharpness}
Let \(1\leq p\leq2\).  If
\(\sup_{z\in\mathbb C\setminus\mathbb R} \|R_0(z)\|_{\mathcal B(\ell^p,\ell^{p'})}<\infty\),
then
\(p\leq\frac{2(d+2)}{d+5}\).
For \(0<\delta<1\), if the supremum is required only on
\(D_\delta\setminus\mathbb R\), the same necessary condition holds when
\(d\geq5\) is odd, whereas
\(p\leq\frac{2(2d+5)}{2d+11}\)
is necessary when \(d\geq4\) is even.  Consequently, the global endpoint
ranges in Theorem~\ref{thm:resolvent} are sharp for \(d\geq5\), and the
away-from-threshold endpoint ranges are sharp for \(d\geq5\) whenever
\(0<\delta<1\).  In dimension \(d=4\), the global power range is sharp;
no assertion is made that the logarithmic endpoint loss is necessary.
\end{theorem}

\begin{remark}\label{rem:d4-log}
The square-root logarithmic loss at the endpoint \(p=4/3\) in
Theorem~\ref{thm:resolvent} is the loss produced by the general
Fourier-decay argument used here.  Removing it would require an additional
endpoint analysis beyond the local Fourier-decay bounds proved in this
paper.  Since such an analysis is not included, no lossless
\(\ell^{4/3}(\mathbb Z^4)\)-to-\(\ell^4(\mathbb Z^4)\) resolvent
estimate is claimed.
\end{remark}

\subsection*{Proof strategy and organization}

The resolvent part of the argument uses two standard reductions.  On a
regular frequency patch, the Fourier decay of the localized level-set
measure is converted into a weighted resolvent estimate by the
regular-level argument of Cuenin \cite{Cuenin2019} and Taira
\cite{Taira2021}.  A pure-power decay exponent \(k>0\) yields the weighted
range \(r\leq2+2k\), and the corresponding unweighted exponent is
determined by \(1/p=1/2+1/r\).  Near a critical point, this reduction is not
applicable because the gradient of the defining function vanishes.
All critical points of the discrete dispersion relation are
nondegenerate, so they are instead controlled by Taira's
critical-point resolvent estimate, based on stationary phase and the
endpoint Strichartz argument of Keel and Tao \cite{KeelTao1998}.
Accordingly, the principal task in this paper is the local
Fourier-decay theorem, Theorem~\ref{thm:local-decay}.

On a regular patch where all sine coordinates are nonzero, we use the
cosine variables \(c_j=\cos(2\pi\xi_j)\).  In these coordinates the
level-set equation becomes an affine hyperplane constraint.  The inverse
coordinate branch has a nonvanishing quadratic term when
\(c_j^0\neq0\), while its first nonlinear term is cubic when
\(c_j^0=0\).  A reduction from general amplitudes to product amplitudes,
followed by a Fourier representation of the hyperplane constraint,
reduces the resulting surface integral to products of one-dimensional
quadratic and cubic oscillatory integrals.  For \(m\geq5\), H\"older's
inequality can be applied above the borderline exponent of the cubic
factors and gives the pure-power decay in
Theorem~\ref{thm:local-decay}.  For \(m\leq4\), the borderline
\(L^4\) estimate produces the stated logarithmic factors.  Patches on
which some sine coordinates vanish are treated by mixed local
coordinates and a normal versus nonstationary decomposition.  Near a
threshold critical point, a dyadic rescaling reduces the regular part
of the singular level set to a uniformly controlled family of
quadratic hypersurfaces.

The necessary conditions are obtained from anisotropic frequency boxes
on which the symbol remains within \(O(\varepsilon)\) of a fixed
energy.  The normal direction has scale \(\varepsilon\), while cubic
and quadratic directions have scales \(\varepsilon^{1/3}\) and
\(\varepsilon^{1/2}\), respectively.  The resulting box volumes give
the endpoint obstructions in Theorem~\ref{thm:sharpness}.

Section~\ref{sec:applications} derives the applications of the main
theorems.  Section~\ref{sec:resolvent} records the resolvent reductions.
Section~\ref{sec:basic} establishes the product-amplitude reduction and
the required nonstationary-phase estimate.  Section~\ref{sec:1d} proves
the one-dimensional cubic and quadratic bounds, and
Section~\ref{sec:separable} combines them into separable hyperplane
estimates.  Section~\ref{sec:cos} develops the local coordinate systems,
proves the Fourier-decay estimates, and derives the sufficient resolvent
bounds.  Finally, Section~\ref{sec:sharp} proves the necessary conditions
and the additional fixed-energy obstructions.

\section{Applications}\label{sec:applications}

We record several consequences of Theorems~\ref{thm:local-decay} and
\ref{thm:resolvent}: Birman--Schwinger bounds and weak-coupling
stability, Kato smoothing, complex-eigenvalue estimates, and improved
thin spectral projection estimates relevant to the Anderson model.

\subsection{Birman--Schwinger bounds and weakly coupled potentials}

For a complex-valued function \(V\) on \(\mathbb Z^d\), define
\[
V^{1/2}(n)
=
\begin{cases}
V(n)/|V(n)|^{1/2},&V(n)\neq0,\\
0,&V(n)=0.
\end{cases}
\]
Thus \(V=|V|^{1/2}V^{1/2}\).

\begin{proposition}[Birman--Schwinger bounds]
\label{prop:birman-schwinger-application}
Suppose that either
\[
d=4,\qquad 1\leq s<2,
\]
or
\[
d\geq5,\qquad
1\leq s\leq\frac{d+2}{3}.
\]
Then, for every \(V\in\ell^s(\mathbb Z^d)\),
\[
\sup_{z\in\mathbb C\setminus\mathbb R}
\left\|
|V|^{1/2}R_0(z)V^{1/2}
\right\|_{\mathcal B(\ell^2)}
\leq
C\|V\|_{\ell^s}.
\]
If \(V\) is real-valued, then there exists \(\gamma_0>0\) such that
\(H_0+\gamma V\) is unitarily equivalent to \(H_0\) whenever
\(\gamma\in\mathbb R\) and \(|\gamma|<\gamma_0\).
\end{proposition}

\begin{proof}
Apply the weighted estimate in Theorem~\ref{thm:resolvent} with
\[
W_1=|V|^{1/2},
\qquad
W_2=V^{1/2},
\qquad
r=2s.
\]
Since
\[
\|W_1\|_{\ell^{2s}}
\|W_2\|_{\ell^{2s}}
=
\|V\|_{\ell^s},
\]
the asserted Birman--Schwinger bound follows.  For real-valued \(V\),
the final assertion follows from the standard weak-coupling criterion;
see \cite[Lemma~1.7]{TadanoTaira2019}.
\end{proof}

For \(d\geq5\), the exponent \((d+2)/3\) is optimal among assumptions
of the form \(V\in\ell^s\).  Indeed, if \(s>(d+2)/3\), choose
\[
\frac{d+2}{3}<q<s.
\]
By \cite[Theorem~1.11(iii)]{TadanoTaira2019}, there is a potential
\(V\in\ell^{q,\infty}\) for which the uniform Birman--Schwinger bound
fails.  Since
\[
\ell^{q,\infty}\subset\ell^s,
\]
the bound cannot hold for all \(V\in\ell^s\).

\begin{proposition}[Kato smoothing]\label{prop:kato-smoothing}
Suppose that
\[
d=4,\qquad 1\leq r<4,
\]
or
\[
d\geq5,\qquad
1\leq r\leq\frac{2(d+2)}{3}.
\]
Then, for \(W\in\ell^r(\mathbb Z^d)\),
\[
\int_{\mathbb R}
\left\|
W e^{-itH_0}f
\right\|_{\ell^2}^2\,dt
\leq
C
\|W\|_{\ell^r}^2
\|f\|_{\ell^2}^2.
\]
\end{proposition}

\begin{proof}
Apply the weighted estimate in Theorem~\ref{thm:resolvent} with
\(W_1=W\) and \(W_2=\overline W\), and then use the standard Kato
smoothing criterion; see \cite{KatoYajima1989}.
\end{proof}

\subsection{Eigenvalue bounds for complex potentials}

The following is the discrete analogue of the usual deduction of
complex eigenvalue bounds from uniform Sobolev estimates; compare
\cite{Frank2011}.

\begin{proposition}[Location of complex eigenvalues]
\label{prop:complex-eigenvalues}
Let \(d\geq5\), set
\[
s_0=\frac{d+2}{3},
\]
and suppose that \(s>s_0\) and \(V\in\ell^s(\mathbb Z^d)\).  If
\(z\in\mathbb C\setminus[0,4d]\) is an eigenvalue of \(H_0+V\), then
\[
\operatorname{dist}(z,[0,4d])^{s-s_0}
\leq
C_{d,s}\|V\|_{\ell^s}^{\,s}.
\]
At the endpoint \(s=s_0\), there is a constant \(c_d>0\) such that
\[
\|V\|_{\ell^{s_0}}<c_d
\]
excludes all eigenvalues of \(H_0+V\) in
\(\mathbb C\setminus[0,4d]\).
\end{proposition}

\begin{proof}
Let
\[
\delta_z=\operatorname{dist}(z,[0,4d])
\]
and define \(p_s\) by
\[
\frac1{p_s}
=
\frac12+\frac1{2s}.
\]
Interpolating the endpoint estimate in
Theorem~\ref{thm:resolvent} with
\[
\|R_0(z)\|_{\mathcal B(\ell^2)}
\leq
\delta_z^{-1}
\]
gives
\[
\|R_0(z)\|_{\mathcal B(\ell^{p_s},\ell^{p_s'})}
\leq
C_{d,s}
\delta_z^{-(1-s_0/s)}.
\]
H\"older's inequality therefore yields
\[
\left\|
|V|^{1/2}R_0(z)V^{1/2}
\right\|_{\mathcal B(\ell^2)}
\leq
C_{d,s}
\delta_z^{-(1-s_0/s)}
\|V\|_{\ell^s}.
\]
If \(z\) is an eigenvalue, the Birman--Schwinger principle implies that
the left-hand side is at least \(1\).  Raising the resulting inequality
to the power \(s\) proves the claim.  The endpoint assertion follows
directly from Proposition~\ref{prop:birman-schwinger-application}.
\end{proof}

\begin{remark}
In dimension \(d=4\), the logarithmic endpoint estimate in
Theorem~\ref{thm:resolvent} implies that
every eigenvalue \(z\in\mathbb C\setminus[0,16]\) of \(H_0+V\), with
\(V\in\ell^2(\mathbb Z^4)\), satisfies
\[
1
\leq
C\|V\|_{\ell^2}
\left(
\log\left(
2+\frac1{\operatorname{dist}(z,[0,16])}
\right)
\right)^{1/2}.
\]
Thus, for small \(\|V\|_{\ell^2}\), any such eigenvalue must lie
exponentially close to the free spectrum.
\end{remark}

\subsection{Thin spectral projections and the Anderson model}

Set
\[
\mathcal E_d
=
\{0,4,\dots,4d\}\cup\{2d\}
\]
and
\[
q_d^\ast
=
\begin{cases}
4,&d=4,\\[3pt]
\dfrac{2(2d+5)}{2d-1},&d\geq5.
\end{cases}
\]

\begin{proposition}[Thin spectral projections]
\label{prop:thin-spectral-projections}
Let
\[
I\Subset(0,4d)\setminus\mathcal E_d.
\]
For every \(q>q_d^\ast\), there exist \(C_{I,q}\) and
\(\alpha_I>0\) such that
\[
\left\|
\mathbf 1_{[E-\alpha,E+\alpha]}(H_0)
\right\|_{\mathcal B(\ell^2,\ell^q)}
\leq
C_{I,q}\alpha^{1/2}
\]
whenever \(E\in I\) and \(0<\alpha<\alpha_I\).
\end{proposition}

\begin{proof}
If a regular point has \(m=d\), then all its cosine coordinates vanish
and its energy is \(2d\).  Hence \(m\leq d-1\) uniformly for level
surfaces with energy in \(I\).

For \(d\geq5\), Theorem~\ref{thm:local-decay}, together with the
uniformity argument in the proof of Proposition~\ref{prop:global-decay},
gives, up to an arbitrarily small loss when a logarithmic factor is
present, the uniform decay exponent
\[
k=\frac{2d-1}{6}.
\]
In dimension \(d=4\), the same argument gives every exponent \(k<1\).

Since \(I\) is compactly contained in the complement of the threshold
energies, \(|\nabla h_0|\) is uniformly bounded below on
\(\bigcup_{\lambda\in I}M_\lambda\).  The level surfaces therefore form a
uniformly smooth compact family, and in particular
\[
\sup_{\lambda\in I}
\sup_{\xi\in\mathbb T^d,\;r>0}
r^{1-d}
\sigma_{M_\lambda}\bigl(M_\lambda\cap B(\xi,r)\bigr)
<\infty.
\]
The lattice Tomas--Stein argument
\cite[Proposition~3.4]{BlackDroginHernandez2025} then gives
\[
\left\|
\int_{M_\lambda}
e^{2\pi i x\cdot\xi}
g(\xi)\,d\sigma_{M_\lambda}(\xi)
\right\|_{\ell^q_x}
\leq
C_{I,q}
\|g\|_{L^2(M_\lambda)}
\]
for every \(q>q_d^\ast\), uniformly for \(\lambda\in I\).
The asserted spectral projection estimate now follows from the coarea
formula and Cauchy--Schwarz in the energy variable.
\end{proof}

The sequence-space inclusions and Theorem~\ref{thm:resolvent} also give
\[
\sup_{z\in\mathbb C\setminus\mathbb R}
\|R_0(z)\|_{\mathcal B(\ell^1,\ell^4)}
<\infty,
\qquad d\geq5,
\]
while in dimension four,
\[
\|R_0(z)\|_{\mathcal B(\ell^1,\ell^4)}
\leq
C\left(
\log\left(
2+\frac1{\operatorname{dist}(z,[0,16])}
\right)
\right)^{1/2}.
\]

These estimates are relevant to the weak-disorder Anderson model.
With the normalization in \cite{BlackDroginHernandez2025},
\[
H_0=2d-\Delta_{\mathbb Z^d},
\]
and the exceptional set \(\mathcal E_d\) corresponds to the bulk
energies excluded there.  The spectral projection exponents used in
that work are
\[
p_4=6,
\qquad
p_d=\frac{2d}{d-3}
\quad(d>4),
\]
whereas Proposition~\ref{prop:thin-spectral-projections} gives every
\[
q>q_d^\ast.
\]
In particular,
\[
q_d^\ast<p_d
\]
in every \(d\geq4\).  Combined with the finite-volume comparison and
transfer argument in
\cite[Sections~3.1--3.2]{BlackDroginHernandez2025}, this gives a sharper
deterministic Tomas--Stein input for their random-resolvent estimates.
As observed in that work, an improvement of this exponent leads to
improved powers of the disorder parameter in dimensions \(d>3\).  We
do not pursue the resulting optimization of the probabilistic
exponents here.

\section{Resolvent reductions}\label{sec:resolvent}

\begin{lemma}[Weighted and unweighted estimates]\label{lem:weighted-unweighted}
Let \(A\) be a linear operator on finitely supported sequences, let
\(1\leq p\leq2\), and let \(2\leq r\leq\infty\) satisfy
\(\frac1p=\frac12+\frac1r\).
Then
\(\|A\|_{\mathcal B(\ell^p,\ell^{p'})}\leq C\)
if and only if
\(\|W_1AW_2\|_{\mathcal B(\ell^2)} \leq C\|W_1\|_{\ell^r}\|W_2\|_{\ell^r}\)
for all finitely supported \(W_1,W_2\).
\end{lemma}

\begin{proof}
The forward implication follows twice from H\"older's inequality:
\[
\|W_1AW_2u\|_2
\leq
\|W_1\|_r\|AW_2u\|_{p'}
\leq
C\|W_1\|_r\|W_2\|_r\|u\|_2.
\]
Conversely, every finitely supported \(f\in\ell^p\) can be written as
\[
f=Wu,
\qquad
W=|f|^{p/r},
\qquad
u=\operatorname{sgn}(f)|f|^{p/2},
\]
with the evident interpretation when \(r=\infty\), and
\(\|W\|_r\|u\|_2=\|f\|_p\).
Factoring both entries in \(\langle Af,g\rangle\) in this way and using the
weighted estimate proves the reverse implication by duality.
\end{proof}

We use the following standard regular-level reduction; see
\cite[Proposition~A.5]{Cuenin2019} and
\cite[Proposition~2.3]{Taira2021}.  The formulation below includes the
parameter-uniform version obtained by the same proof.

\begin{proposition}[Fourier decay implies resolvent bounds]
\label{prop:fourier-to-resolvent}
Let \(T\in C^\infty(\mathbb T^d;\mathbb R)\), and let \(\chi\) be supported
where \(\nabla T\neq0\).  Suppose that there exists \(k>0\) such that,
uniformly for the level surfaces meeting \(\operatorname{supp}\chi\) and
for smooth amplitudes in a bounded subset of \(C^\infty\),
\[
\left|
\int_{\{T=\lambda\}}
e^{-2\pi i x\cdot\xi}a(\xi)\,d\sigma_\lambda(\xi)
\right|
\leq
C(1+|x|)^{-k},
\qquad x\in\mathbb Z^d.
\]
Then, for \(1\leq r\leq2+2k\),
\[
\sup_{z\in\mathbb C\setminus\mathbb R}
\|W_1\chi(D)(T(D)-z)^{-1}W_2\|_{\mathcal B(\ell^2)}
\leq
C\|W_1\|_{\ell^r}\|W_2\|_{\ell^r}.
\]
The same conclusion holds uniformly on any compact parameter family for
which the hypotheses are uniform.
\end{proposition}

On a regular patch the coarea density \(1/|\nabla T|\) is smooth and may
be absorbed into the amplitude.  If the Fourier bound contains a factor
\((\log(2+|x|))^b\), Proposition~\ref{prop:fourier-to-resolvent} applies
with every exponent smaller than \(k\).

\begin{proposition}[The logarithmic endpoint]\label{prop:log-endpoint}
Assume the hypotheses of Proposition~\ref{prop:fourier-to-resolvent} and
suppose, in dimension four, that, for some \(A\geq1\),
\[
\left|
\int_{\{T=\lambda\}}
e^{-2\pi i x\cdot\xi}a(\xi)\,d\sigma_\lambda(\xi)
\right|
\leq
A(1+|x|)^{-1}\log(2+|x|),
\qquad x\in\mathbb Z^4.
\]
Then
\[
\|\chi(D)(T(D)-z)^{-1}\|_{\mathcal B(\ell^{4/3},\ell^4)}
\leq
CA^{1/2}
\left(
\log\left(
2+\frac1{\operatorname{dist}(z,T(\mathbb T^4))}
\right)
\right)^{1/2}.
\]
The equivalent weighted estimate holds with \(W_1,W_2\in\ell^4\).
\end{proposition}

\begin{proof}
Let \(0<\eta\leq1/2\).  Since
\(\log(2+s)\leq C\eta^{-1}(1+s)^\eta\),
the assumed estimate implies pure power decay of order \(1-\eta\), with
constant \(CA\eta^{-1}\).  Tracking the constant in the complex
interpolation proof of Proposition~\ref{prop:fourier-to-resolvent} gives
\[
\|\chi(D)(T(D)-z)^{-1}\|_
{\mathcal B(\ell^{p_\eta},\ell^{p_\eta'})}
\leq
C A^{1/(2-\eta)}\eta^{-1/(2-\eta)},
\qquad
p_\eta=\frac{2(2-\eta)}{3-\eta}.
\]
Indeed, the kernel bound on the upper boundary of the analytic family is
linear in \(A\eta^{-1}\), and the interpolation parameter at the
resolvent is \(1/(2-\eta)\); see the proof of
\cite[Proposition~A.5]{Cuenin2019}.

The spectral theorem also gives
\[
\|\chi(D)(T(D)-z)^{-1}\|_{\mathcal B(\ell^2)}
\leq
\operatorname{dist}(z,T(\mathbb T^4))^{-1}.
\]
Interpolate the last two estimates with weight \(\eta/2\) at the
\(\ell^2\) endpoint.  The identity
\(\frac34 = \left(1-\frac\eta2\right) \frac{3-\eta}{2(2-\eta)} +\frac\eta2\cdot\frac12\)
gives
\[
\|\chi(D)(T(D)-z)^{-1}\|_{\mathcal B(\ell^{4/3},\ell^4)}
\leq
CA^{1/2}\eta^{-1/2}
\operatorname{dist}(z,T(\mathbb T^4))^{-\eta/2}.
\]
Taking
\(\eta = \min\left\{ \frac12,\, \frac1{\log(2+\operatorname{dist}(z,T(\mathbb T^4))^{-1})} \right\}\)
proves the assertion.  Lemma~\ref{lem:weighted-unweighted} gives the
weighted formulation.
\end{proof}

\begin{proposition}[Nondegenerate critical points]
\label{prop:critical-resolvent}
Let \(T\in C^\infty(\mathbb T^d;\mathbb R)\), \(d\geq3\), and suppose that
\(\xi^0\) is a nondegenerate critical point of \(T\).  If \(\chi\) is
supported sufficiently close to \(\xi^0\), then
\[
\sup_{z\in\mathbb C\setminus\mathbb R}
\|W_1\chi(D)(T(D)-z)^{-1}W_2\|_{\mathcal B(\ell^2)}
\leq
C\|W_1\|_{\ell^r}\|W_2\|_{\ell^r}
\]
for \(1\leq r\leq d\).
\end{proposition}

\begin{proof}
Choose a real-valued \(\widetilde\chi\in C_c^\infty(\mathbb T^d)\) such that
\(\widetilde\chi=1\) on \(\operatorname{supp}\chi\) and
\(\widetilde\chi\) is supported in the same sufficiently small
neighborhood.  Stationary phase gives
\(\|\widetilde\chi(D)^2e^{-itT(D)}\|_{\mathcal B(\ell^1,\ell^\infty)}
\leq C\langle t\rangle^{-d/2}\).
As pointed out by Taira \cite{Taira2021}, the endpoint \(TT^*\) argument of Keel and Tao \cite[Theorem 1.2]{KeelTao1998} then
gives, on every finite interval \(I\),
\[
\|\widetilde\chi(D)^2u\|_{L_t^2(I;\ell^{2d/(d-2)})}
\leq
C\|u(t_0)\|_{\ell^2}
+
C\|(i\partial_t-T(D))u\|_{L_t^2(I;\ell^{2d/(d+2)})},
\]
where \(t_0\in I\) is an endpoint of \(I\), and the constant is
independent of \(I\).  For \(\operatorname{Im}z>0\), take \(I=[0,S]\),
choose \(t_0=0\), and apply this estimate to
\[u(t)=e^{-itz}(T(D)-z)^{-1}\chi(D)f.\]  Since
\[(i\partial_t-T(D))u=-e^{-itz}\chi(D)f\] and
\(\widetilde\chi(D)^2u=u\), division by
\(\|e^{-itz}\|_{L^2(0,S)}\) and passage to the limit \(S\to\infty\)
yield
\[\|\chi(D)(T(D)-z)^{-1}f\|_{\ell^{2d/(d-2)}}
\leq C\|\chi(D)f\|_{\ell^{2d/(d+2)}}
\leq C\|f\|_{\ell^{2d/(d+2)}}.\]
The lower half-plane follows by taking adjoints and applying the same
argument with \(\overline\chi\).  Thus
\[\sup_{z\in\mathbb C\setminus\mathbb R}
\|\chi(D)(T(D)-z)^{-1}\|_{
\mathcal B(\ell^{2d/(d+2)},\ell^{2d/(d-2)})}\leq C.\]
Lemma~\ref{lem:weighted-unweighted} gives the assertion for \(r=d\).
The remaining values \(1\leq r\leq d\) follow from
\(\ell^r\subset\ell^d\).
\end{proof}

\section{Basic reductions}\label{sec:basic}

For \(n\geq2\), set
\[
\Pi_n=
\left\{
y=(y_1,\dots,y_n)\in\mathbb R^n:
\sum_{j=1}^n y_j=0
\right\}.
\]
Let \(\Pi_n^*\) denote the dual space of \(\Pi_n\).
We write \(dy_\Pi=d\sigma_{\Pi_n}(y)\) for the Euclidean surface
measure on \(\Pi_n\).  By the coarea formula,
\[
\int_{\Pi_n}F(y)\,dy_\Pi
=
\sqrt{n}
\int_{\mathbb R^n}
F(y)\,
\delta\left(\sum_{j=1}^n y_j\right)\,dy.
\]
Equivalently, for compactly supported smooth \(F\),
\[
\int_{\Pi_n}F(y)\,dy_\Pi
=
\frac{\sqrt{n}}{2\pi}
\int_{\mathbb R}
\int_{\mathbb R^n}
F(y)e^{is\sum_{j=1}^n y_j}\,dy\,ds,
\]
where the identity is understood by first regularizing the delta mass and
then passing to the limit.

\begin{lemma}[Reduction to product cutoffs on \(\Pi_n\)]\label{lem:reduction}
Let \(n\geq2\), and let
\(\mathcal B:[1,\infty)\to(0,\infty)\)
be a positive function. Let \(\Phi_R\) be a real-valued smooth phase on \(\Pi_n\),
depending on \(R\geq1\). Suppose that, for every compact interval
\(K\subset\mathbb R\), there exist \(M\) and \(C_K\) such that for every product
amplitude
\(A(y)=\prod_{j=1}^n\psi_j(y_j), \psi_j\in C_c^\infty(K)\),
one has
\[
\sup_{\ell\in \Pi_n^*}
\left|
\int_{\Pi_n}
e^{i\Phi_R(y)+i\ell(y)}
\prod_{j=1}^n\psi_j(y_j)\,dy_\Pi
\right|
\leq
C_K\mathcal B(R)
\prod_{j=1}^n
\max_{0\leq r\leq M}\|\psi_j^{(r)}\|_{L^\infty}.
\]
Then for every \(a\in C_c^\infty(\Pi_n)\),
\[
\sup_{\ell\in \Pi_n^*}
\left|
\int_{\Pi_n}
e^{i\Phi_R(y)+i\ell(y)}a(y)\,dy_\Pi
\right|
\leq
C_a\mathcal B(R),
\]
where \(C_a\) depends on finitely many seminorms of \(a\), but not on \(\ell\) or
\(R\).
\end{lemma}

\begin{proof}
Let
\(P_\Pi y=y-\frac1n\left(\sum_{j=1}^ny_j\right)(1,\dots,1)\)
be the orthogonal projection onto \(\Pi_n\). Choose \(\zeta\in C_c^\infty(\mathbb R)\)
with \(\zeta(0)=1\), and extend \(a\) to
\(A(y)=a(P_\Pi y)\zeta\left(\sum_{j=1}^ny_j\right)\in C_c^\infty(\mathbb R^n)\).
Then \(A|_{\Pi_n}=a\). Choose a cube \(Q=\prod_{j=1}^n I_j\) of side length \(L\), with
\(\operatorname{supp}A\Subset Q\), and periodize \(A\) with period \(L\)
in each variable. Since \(A\) vanishes
near \(\partial Q\), the periodic extension is smooth, and on \(Q\)
\[
A(y)=\sum_{\nu\in\mathbb Z^n}c_\nu e^{2\pi i\nu\cdot y/L},
\qquad
|c_\nu|\leq C_N(1+|\nu|)^{-N}
\]
for every \(N\). Choose \(\chi_j\in C_c^\infty(I_j)\) such that
\(\prod_{j=1}^n\chi_j(y_j)=1\)
on \(\operatorname{supp}A\). Then
\[
A(y)=
\sum_{\nu\in\mathbb Z^n}c_\nu
\prod_{j=1}^n\chi_j(y_j)e^{2\pi i\nu_jy_j/L}
\]
on the support of the product cutoff. The exponential factor only changes the linear
perturbation on \(\Pi_n\), and
\[
\max_{0\leq r\leq M}
\left\|
\partial_t^r\left(\chi_j(t)e^{2\pi i\nu_jt/L}\right)
\right\|_{L^\infty}
\leq C_{M,Q}(1+|\nu_j|)^M.
\]
Choosing \(N\) large enough and summing over \(\nu\) proves the claim.
\end{proof}

\begin{lemma}[Uniform nonstationary phase]
\label{lem:nonstationary}
Let \(E\) be a fixed finite-dimensional real Euclidean space, let
\(a\in C_c^\infty(E)\), let \(R>0\), and let
\(\Phi\in C^\infty(E;\mathbb R)\). Assume that, on a fixed open
neighborhood of \(\operatorname{supp}a\),
\(|\nabla_E\Phi(y)|\geq cR\),
and, for every multi-index \(\alpha\) with \(|\alpha|\geq1\),
\(|\partial^\alpha\Phi(y)|\leq C_\alpha R\).
Here the derivatives are taken in any fixed orthonormal linear coordinates
on \(E\).

Then, for every integer \(N\geq1\),
\[
\left|
\int_E e^{i\Phi(y)}a(y)\,dy
\right|
\leq
C_N R^{-N}
\sum_{|\alpha|\leq N}
\|\partial^\alpha a\|_{L^1(E)}.
\]
The constant \(C_N\) depends only on \(N\), \(c\), \(E\), and finitely many
of the constants \(C_\alpha\), but not on \(R\) or on any additional
parameters on which \(\Phi\) may depend.
\end{lemma}

\begin{proof}
This proof is a modified version of Stein \cite[Chapter VIII, Proposition 3]{Stein1993}.
Set
\(\varphi=R^{-1}\Phi, V=\frac{\nabla_E\varphi}{|\nabla_E\varphi|^2}\).
After multiplying \(V\) by a fixed cutoff supported in the neighborhood
appearing in the hypotheses and equal to one near
\(\operatorname{supp}a\), we may regard \(V\) as a smooth vector field on
\(E\). The assumptions and the chain rule imply that, for every
multi-index \(\gamma\),
\(\sup_{\operatorname{supp}a}|\partial^\gamma V| \leq C_\gamma\).

Define
\(L=\frac1{iR}V\cdot\nabla_E\).
Then
\(Le^{i\Phi}=e^{i\Phi}\)
near \(\operatorname{supp}a\). With respect to the bilinear integration
pairing, the formal transpose of \(L\) is
\(L^{\mathrm t}f = -\frac1{iR}\operatorname{div}_E(Vf)\).
Thus, writing
\(\mathcal Df=\operatorname{div}_E(Vf)\),
and integrating by parts \(N\) times,
\(\int_E e^{i\Phi}a\,dy = \left(-\frac1{iR}\right)^N \int_E e^{i\Phi}\mathcal D^Na\,dy\).
The uniform bounds for the derivatives of \(V\) give
\(\|\mathcal D^Na\|_{L^1(E)} \leq C_N \sum_{|\alpha|\leq N} \|\partial^\alpha a\|_{L^1(E)}\).
The result follows.
\end{proof}

\section{One-dimensional estimates}\label{sec:1d}

\begin{lemma}[{One-dimensional van der Corput estimates, \cite[Chapter VIII, Proposition 3]{Stein1993}}]\label{lem:vandercorput}
Let \(I\subset\mathbb R\) be an interval, let \(b\in C_c^1(I)\), and let
\(\phi\in C^3(I;\mathbb R)\).

If \(k=2\) or \(k=3\), \(\phi^{(k)}\) has a fixed sign on \(I\), and
\(|\phi^{(k)}(t)|\geq\lambda>0\)
on \(I\), then
\[
\left|\int_I e^{i\phi(t)}b(t)\,dt\right|
\leq
C_k\lambda^{-1/k}
\left(\|b\|_{L^\infty}+\|b'\|_{L^1}\right).
\]

If \(\phi'\) has a fixed sign on \(I\), \(|\phi'(t)|\geq\lambda>0\) on \(I\), and
\(\phi'\) is monotone on \(I\), then
\[
\left|\int_I e^{i\phi(t)}b(t)\,dt\right|
\leq
C\lambda^{-1}
\left(\|b\|_{L^\infty}+\|b'\|_{L^1}\right).
\]
\end{lemma}



\begin{lemma}[One-dimensional cubic estimate]\label{lem:cubic-1d}
Let \(q\in C^\infty((-r_0,r_0))\) be real-valued and satisfy
\[
q(0)=q'(0)=q''(0)=0,
\qquad
|q'''(t)|\geq c_0>0
\]
on \((-r_0,r_0)\). Let \(\psi\in C_c^\infty((-r_0,r_0))\). For \(T\geq1\), define
\[B_T(u)= \int_{\mathbb R}e^{i(T^3q(x/T)+ux)}\psi(x/T)\,dx.\]
Then
\[
|B_T(u)|\leq C(1+|u|)^{-1/4}.
\]
Moreover, for every \(N\geq1\),
\(|B_T(u)|\leq C_N(1+|u|)^{-N}\)
whenever \(|u|>CT^2\). Consequently,
\(\|B_T\|_{L^p(\mathbb R)}\leq C_p \quad (p>4)\),
and
\(\|B_T\|_{L^4(\mathbb R)}\leq C\bigl(\log(e+T)\bigr)^{1/4}\).

All constants are independent of \(T\) and \(u\). They may depend on
\(c_0\), on \(q\) and \(\psi\) through finitely many seminorms, and on
\(N\) or \(p\) when indicated.
\end{lemma}

\begin{proof}
Let
\(\phi_{T,u}(x)=T^3q(x/T)+ux\).
Then
\[
\phi_{T,u}'(x)=T^2q'(x/T)+u,
\qquad
\phi_{T,u}''(x)=Tq''(x/T),
\qquad
\phi_{T,u}'''(x)=q'''(x/T).
\]
Moreover,
\(\|\psi(x/T)\|_{L^\infty} + \|\partial_x(\psi(x/T))\|_{L^1} \leq C\).
Since
\(|\phi_{T,u}'''(x)|\geq c_0\)
on the support of \(\psi(x/T)\), the third-order estimate in
Lemma~\ref{lem:vandercorput} gives
\[
|B_T(u)|
\leq
C
\left(
\|\psi(x/T)\|_{L^\infty}
+
\|\partial_x(\psi(x/T))\|_{L^1}
\right)
\leq C.
\]

Since \(q'''\) is continuous and nonvanishing on the connected interval
\((-r_0,r_0)\), it has a fixed sign.  The conditions
\(q'(0)=q''(0)=0\) therefore imply
\[
c|t|^2\leq |q'(t)|\leq C|t|^2,
\qquad
c|t|\leq |q''(t)|\leq C|t|
\]
on the convex hull of \(\{0\}\cup\operatorname{supp}\psi\). It remains to prove the bound
\(|B_T(u)|\leq C|u|^{-1/4}\)
for \(|u|\geq2\), since the case \(|u|<2\) follows from the uniform bound already proved.

Set
\(R=|u|^{1/2}\).
Choose constants \(0<c_1<C_1\), depending only on the constants in the
preceding comparison estimates, so that zeros of
\(\phi_{T,u}'(x)=T^2q'(x/T)+u\)
can occur only in the region
\(3c_1R\leq |x|\leq \frac{C_1}{3}R\).
Such constants exist because
\(c|x|^2\leq T^2|q'(x/T)|\leq C|x|^2\).
Choose smooth functions \(\eta_0,\eta_1,\eta_\infty\) on \([0,\infty)\) such that
\(\eta_0(r)+\eta_1(r)+\eta_\infty(r)=1\),
\[
\operatorname{supp}\eta_0\subset[0,2c_1],
\qquad
\operatorname{supp}\eta_1\subset[c_1,C_1],
\qquad
\operatorname{supp}\eta_\infty\subset[C_1/2,\infty).
\]
Decompose
\(B_T(u)=B_0(u)+B_1(u)+B_\infty(u)\),
where
\[
B_\nu(u)
=
\int_{\mathbb R}
e^{i\phi_{T,u}(x)}
\eta_\nu(|x|/R)\psi(x/T)\,dx,
\qquad
\nu\in\{0,1,\infty\}.
\]

On the support of \(B_1(u)\), one has
\(|x|\simeq R\).
Splitting the integral into the two half-lines \(x>0\) and \(x<0\), the second derivative
has fixed sign on each piece and satisfies
\(|\phi_{T,u}''(x)| = T|q''(x/T)| \simeq |x| \simeq R\).
The corresponding amplitudes satisfy
\[
\|\eta_1(|x|/R)\psi(x/T)\|_{L^\infty}
+
\|\partial_x(\eta_1(|x|/R)\psi(x/T))\|_{L^1}
\leq C.
\]
Therefore the second-order estimate in
Lemma~\ref{lem:vandercorput} gives
\[
|B_1(u)|\leq CR^{-1/2}=C|u|^{-1/4}.
\]

On the supports of \(B_0(u)\) and \(B_\infty(u)\), the choice of
\(c_1\) and \(C_1\) gives
\[
|\phi_{T,u}'(x)|\geq c|u|.
\]
After splitting into the two half-lines \(x>0\) and \(x<0\),
\(\phi_{T,u}'\) is monotone on each piece.  The first-derivative estimate
in Lemma~\ref{lem:vandercorput} gives
\[
|B_0(u)|+|B_\infty(u)|
\leq C|u|^{-1}
\leq C|u|^{-1/4}.
\]
Combining the three pieces gives
\(|B_T(u)| \leq C|u|^{-1/4}\).

If
\[
|u|>CT^2
\]
with \(C\) sufficiently large, then
\[
|\phi_{T,u}'(x)|\geq c|u|
\]
on the whole support of \(\psi(x/T)\), and every positive-order
derivative of \(\phi_{T,u}\) is \(O(|u|)\) there.  Apply
Lemma~\ref{lem:nonstationary} with scale \(|u|\) and with \(N+1\) in
place of \(N\).  Since
\[
\sum_{r=0}^{N+1}
\left\|
\partial_x^r\bigl(\psi(x/T)\bigr)
\right\|_{L^1}
\leq C_NT
\leq C_N|u|^{1/2},
\]
we obtain, for every \(N\geq1\),
\[
|B_T(u)|\leq C_N(1+|u|)^{-N}.
\]

Finally,
\[
\|B_T\|_{L^p}^p
\leq
C\int_{|u|\leq CT^2}(1+|u|)^{-p/4}\,du
+
C_N\int_{|u|>CT^2}(1+|u|)^{-N}\,du.
\]
The right-hand side is bounded uniformly in \(T\) for \(p>4\). For \(p=4\), it is bounded
by
\(C\log(e+T)\).
Taking the \(p\)-th root proves the asserted \(L^p\) estimates.
\end{proof}

\begin{lemma}[One-dimensional quadratic estimates]
\label{lem:quadratic-1d}
Let \(g\in C^\infty((-r_0,r_0))\) be real-valued and satisfy
\(|g''(t)|\geq c_0>0\)
on \((-r_0,r_0)\). Let \(\psi\in C_c^\infty((-r_0,r_0))\). For \(|\tau|\geq1\), define
\(Q_\tau(v)= \int_{\mathbb R}e^{i(\tau g(t)+vt)}\psi(t)\,dt\).
Then, for every \(2\leq q\leq\infty\),
\(\|Q_\tau\|_{L^q(\mathbb R)} \leq C_q|\tau|^{-1/2+1/q}\).
In particular,
\(\|Q_\tau\|_{L^\infty(\mathbb R)} \leq C|\tau|^{-1/2}\).
\end{lemma}

\begin{proof}
Let
\(\phi_{\tau,v}(t)=\tau g(t)+vt\).
Then
\(|\phi_{\tau,v}''(t)| = |\tau|\,|g''(t)| \geq c|\tau|\)
on the support of \(\psi\). Also
\(\|\psi\|_{L^\infty}+\|\psi'\|_{L^1}\leq C\).
The second-order estimate in Lemma~\ref{lem:vandercorput} therefore gives
\[
|Q_\tau(v)|
\leq
C|\tau|^{-1/2}
\left(
\|\psi\|_{L^\infty}+\|\psi'\|_{L^1}
\right)
\leq
C|\tau|^{-1/2}.
\]
Taking the supremum in \(v\) proves the claim.

On the other hand, by Plancherel,
\[
\|Q_\tau\|_{L^2(\mathbb R)}
\leq
C\left\|e^{i\tau g(t)}\psi(t)\right\|_{L^2(\mathbb R)}
=
C\|\psi\|_{L^2(\mathbb R)}
\leq C.
\]
Interpolating between the \(L^2\) and \(L^\infty\) estimates gives
\[
\|Q_\tau\|_{L^q(\mathbb R)}
\leq
C_q|\tau|^{-1/2+1/q},
\qquad
2\leq q\leq\infty.
\]
\end{proof}

\section{Separable hyperplane estimates}\label{sec:separable}

\begin{proposition}[Separable hyperplane estimate]
\label{prop:high-flat}
Let \(d\geq m\geq5\).  Let \(q_j\) and \(p_k\) be real-valued smooth
functions on a fixed interval about the origin.  For \(1\leq j\leq m\),
assume
\[
q_j(0)=q_j'(0)=q_j''(0)=0,
\qquad
|q_j'''(t)|\geq c_0>0.
\]
For \(m+1\leq k\leq d\), assume
\(|p_k''(t)|\geq c_0>0\).
Let \(a\in C_c^\infty(\Pi_d)\) be supported sufficiently close to the origin. Let
\(R\geq1\), and suppose
\(cR\leq |\tau_j|\leq CR, j=1,\dots,d\).
Define
\begin{equation}
\mathcal I(\tau,\beta)
=
\int_{\Pi_d}
\exp\left(i
\left[
\sum_{j=1}^m\tau_jq_j(t_j)
+
\sum_{k=m+1}^d\tau_kp_k(t_k)
+
\sum_{j=1}^d\beta_jt_j
\right]\right)
a(t)\,dt_\Pi.
\tag{$\ast$}\label{eq:separable-integral}
\end{equation}
Then
\(|\mathcal I(\tau,\beta)| \leq CR^{-(d-m)/2-(m-1)/3}\).
The constant is independent of \(\beta\), \(\tau\), and \(R\).
\end{proposition}

\begin{proof}
By Lemma~\ref{lem:reduction} with \(n=d\), it suffices to consider
\(a(t)=\prod_{j=1}^d\psi_j(t_j)\).
Using the Fourier representation of the hyperplane constraint,
\[
\mathcal I(\tau,\beta)
=
\frac{\sqrt{d}}{2\pi}
\int_{\mathbb R}
\prod_{j=1}^m C_j(s+\beta_j)
\prod_{k=m+1}^d Q_k(s+\beta_k)\,ds,
\]
where
\(C_j(r)=\int_{\mathbb R}e^{i(\tau_jq_j(t)+rt)}\psi_j(t)\,dt\),
and
\(Q_k(r)=\int_{\mathbb R}e^{i(\tau_kp_k(t)+rt)}\psi_k(t)\,dt\).
Lemma~\ref{lem:quadratic-1d} gives
\[
\|Q_k\|_{L^\infty}\leq CR^{-1/2}.
\]
For \(1\leq j\leq m\), set \(T_j=|\tau_j|^{1/3}\). With \(x=T_jt\),
\(C_j(r)=T_j^{-1}B_j(r/T_j)\),
where
\(B_j(u)=\int_{\mathbb R}e^{i(\tau_jq_j(x/T_j)+ux)}\psi_j(x/T_j)\,dx\).
Lemma~\ref{lem:cubic-1d}, applied to
\(\operatorname{sgn}(\tau_j)q_j\), gives, for every \(p>4\),
\[
\|C_j\|_{L^p}\leq C_pR^{-1/3+1/(3p)}.
\]
Using H\"older's inequality with all cubic exponents equal to \(m\),
\[
|\mathcal I(\tau,\beta)|
\leq
CR^{-(d-m)/2}
\prod_{j=1}^m\|C_j\|_{L^m}
\leq
CR^{-(d-m)/2-(m-1)/3}.
\]
\end{proof}

\begin{proposition}[Low-flat separable hyperplane estimate]
\label{prop:low-flat}
Let \(d\geq4\) and \(0\leq m\leq4\).  Let \(q_j\) and \(p_k\) be
real-valued smooth functions on a fixed interval about the origin and
assume
\[
q_j(0)=q_j'(0)=q_j''(0)=0,
\qquad
|q_j'''(t)|\geq c_0>0,
\qquad 1\leq j\leq m,
\]
and
\(|p_k''(t)|\geq c_0>0, m+1\leq k\leq d\).
Let \(a\in C_c^\infty(\Pi_d)\) be supported sufficiently close to the
origin, let \(R\geq1\), and assume
\(cR\leq|\tau_j|\leq CR, j=1,\dots,d\).
Define \(\mathcal I(\tau,\beta)\) by
\eqref{eq:separable-integral}. Then
\(|\mathcal I(\tau,\beta)| \leq C R^{-(d-2)/2} \bigl(\log(2+R)\bigr)^{m/4}\).
The constant is independent of \(\beta\), \(\tau\), and \(R\).
\end{proposition}

\begin{proof}
As in the proof of Proposition~\ref{prop:high-flat},
Lemma~\ref{lem:reduction} and the Fourier representation of the
hyperplane constraint give
\[
\mathcal I(\tau,\beta)
=
\frac{\sqrt{d}}{2\pi}
\int_{\mathbb R}
\prod_{j=1}^m C_j(s+\beta_j)
\prod_{k=m+1}^d Q_k(s+\beta_k)\,ds.
\]
For the cubic factors, Lemma~\ref{lem:cubic-1d}, applied to
\(\operatorname{sgn}(\tau_j)q_j\), gives
\[
\|C_j\|_{L^4}
\leq
CR^{-1/4}\bigl(\log(2+R)\bigr)^{1/4}.
\]
For the quadratic factors, Lemma~\ref{lem:quadratic-1d} gives
\[
\|Q_k\|_{L^p}
\leq
C_pR^{-1/2+1/p},
\qquad
2\leq p\leq\infty.
\]

If \(0\leq m\leq3\), set
\(q=\frac{4(d-m)}{4-m}\).
Then \(q\geq2\) and
\(\frac m4+\frac{d-m}{q}=1\).
H\"older's inequality therefore gives
\(|\mathcal I(\tau,\beta)| \leq C \prod_{j=1}^m\|C_j\|_{L^4} \prod_{k=m+1}^d\|Q_k\|_{L^q}\).
Hence
\[
|\mathcal I(\tau,\beta)|
\leq
C
R^{-m/4}
\bigl(\log(2+R)\bigr)^{m/4}
R^{-(d-m)(1/2-1/q)}.
\]
Since
\(\frac{d-m}{q}=1-\frac m4\),
the total power of \(R\) is
\(\frac m4+(d-m)\left(\frac12-\frac1q\right) = \frac{d-2}{2}\).
This proves the claim for \(0\leq m\leq3\).

If \(m=4\), use the four cubic factors in \(L^4\) and all quadratic
factors in \(L^\infty\). Then
\[
|\mathcal I(\tau,\beta)|
\leq
C
\left(
R^{-1/4}
\bigl(\log(2+R)\bigr)^{1/4}
\right)^4
R^{-(d-4)/2},
\]
which is
\(|\mathcal I(\tau,\beta)| \leq C R^{-(d-2)/2} \log(2+R)\).
\end{proof}

\section{Cosine coordinates and the regular coordinate-flat patches}\label{sec:cos}

Let
\[
h_0(\xi)=4\sum_{j=1}^d\sin^2(\pi\xi_j)
=2d-2\sum_{j=1}^d\cos(2\pi\xi_j).
\]
\(M_\lambda=\{\xi\in\mathbb T^d:h_0(\xi)=\lambda\}\)
is equivalently
\(\sum_{j=1}^d\cos(2\pi\xi_j)=d-\frac\lambda2\).
Let \(\xi^0\in M_\lambda\) satisfy
\(\sin(2\pi\xi_j^0)\neq0, j=1,\dots,d\).
Set
\(\theta_j=2\pi\xi_j, c_j=\cos\theta_j\).
Since \(\sin\theta_j^0\neq0\), the map \(\theta_j\mapsto c_j\) is a local diffeomorphism.
Let \(\theta_j=\theta_j(c_j)\) be the corresponding inverse branch. Writing
\(c_j=c_j^0+t_j\),
the surface equation becomes
\(\sum_{j=1}^d t_j=0\).
The Euclidean surface measure in these coordinates is a smooth nonvanishing density on
\(\Pi_d\) times \(dt_\Pi\). Thus, if \(\chi\) is a smooth cutoff supported sufficiently close to
\(\xi^0\), then, up to replacing \(x\) by \(-x\), the localized
surface-measure transform
\(\widehat{\chi\,d\sigma_{M_\lambda}}(x)\) has the form
\[
\int_{\Pi_d}
e^{i\sum_{j=1}^d x_j\theta_j(c_j^0+t_j)}
A(t)\,dt_\Pi,
\]
where \(A\in C_c^\infty(\Pi_d)\) incorporates the pullback of \(\chi\)
and the surface-measure density.

The inverse branch satisfies
\[
\theta_j'(c)=-\frac1{\sin\theta_j(c)},
\qquad
\theta_j''(c)=-\frac{\cos\theta_j(c)}{\sin^3\theta_j(c)}.
\]
Therefore, if \(c_j^0\neq0\), then
\[
\theta_j(c_j^0+t)
=
\theta_j(c_j^0)+\theta_j'(c_j^0)t
+\frac12\theta_j''(c_j^0)t^2+O(t^3),
\qquad
\theta_j''(c_j^0)\neq0.
\]
If \(c_j^0=0\), then \(\theta_j^0\in\{\pi/2,3\pi/2\}\), and
\[
\theta_j(c_j^0+t)=
\theta_j(c_j^0)+\theta_j'(c_j^0)
\left(t+\frac{t^3}{6}+O(t^5)\right).
\]

\begin{proposition}[Regular coordinate-flat decay]
\label{prop:regular-cosine}
Let \(d\geq4\), and let
\(m=\#\{j:\cos(2\pi\xi_j^0)=0\}\).
Assume
\(\sin(2\pi\xi_j^0)\neq0, j=1,\dots,d\).
Let \(\chi\) be supported sufficiently close to \(\xi^0\).

If \(0\leq m\leq4\), then
\[
|\widehat{\chi\,d\sigma_{M_\lambda}}(x)|
\leq
C
(1+|x|)^{-(d-2)/2}
\bigl(\log(2+|x|)\bigr)^{m/4}.
\]
If \(m\geq5\), then
\[
|\widehat{\chi\,d\sigma_{M_\lambda}}(x)|
\leq
C
(1+|x|)^{-(d-m)/2-(m-1)/3}.
\]
\end{proposition}

\begin{proof}
It is enough to consider \(|x|\geq1\). Let
\(\Phi_x(t)=\sum_{j=1}^dx_j\theta_j(c_j^0+t_j)\).
Set
\(b_j=x_j\theta_j'(c_j^0), b=(b_1,\dots,b_d)\).
Since each \(\theta_j'(c_j^0)\) is nonzero and fixed, \(|b|\simeq |x|\).
The tangential gradient
on \(\Pi_d\) of \(\sum_jb_jt_j\) has size
\(\operatorname{dist}\bigl(b,\mathbb R(1,\dots,1)\bigr)\).
If this distance is at least \(\varepsilon |x|\), then, after shrinking
the support,
\(|\nabla_{\Pi_d}\Phi_x(t)|\geq c_\varepsilon |x|\)
on the support of \(A\). Moreover,
\(|\partial^\alpha\Phi_x(t)| \leq C_\alpha|x|, |\alpha|\geq1\).
Lemma~\ref{lem:nonstationary}, applied with
\(E=\Pi_d\) and \(R=|x|\), therefore gives, for every \(N\geq1\),
\(\left| \int_{\Pi_d}e^{i\Phi_x(t)}A(t)\,dt_\Pi \right| \leq C_{N,\varepsilon}|x|^{-N}\).
This is stronger than the asserted estimates.

It remains to consider
\[
\operatorname{dist}\bigl(b,\mathbb R(1,\dots,1)\bigr)<\varepsilon |x|.
\]
Let \(\rho\) be chosen so that
\(|b-\rho(1,\dots,1)|= \operatorname{dist}\bigl(b,\mathbb R(1,\dots,1)\bigr)\).
For \(\varepsilon\) sufficiently small, \(|\rho|\simeq |x|\), and hence
\(|x_j|\simeq |x|, j=1,\dots,d\).
After a permutation, the first \(m\)
coordinates are precisely those for which \(c_j^0=0\).
For \(1\leq j\leq m\), define
\(q_j(t) = \theta_j(c_j^0+t)-\theta_j(c_j^0)-\theta_j'(c_j^0)t\).
Then
\[
q_j(0)=q_j'(0)=q_j''(0)=0,
\qquad
q_j'''(0)=\theta_j'(c_j^0)\neq0,
\]
and, after shrinking the support,
\(|q_j'''(t)|\geq c>0\).
For \(m+1\leq k\leq d\), define
\(p_k(t) = \theta_k(c_k^0+t)-\theta_k(c_k^0)-\theta_k'(c_k^0)t\).
Since \(c_k^0\neq0\),
\(|p_k''(t)|\geq c>0\)
on the support. Set
\[
\gamma_x=\sum_{j=1}^d x_j\theta_j(c_j^0),
\qquad
\tau_j=x_j,
\qquad
\beta_j=x_j\theta_j'(c_j^0)=b_j.
\]
By the definitions of \(q_j\) and \(p_k\),
\[
\begin{aligned}
\int_{\Pi_d}e^{i\Phi_x(t)}A(t)\,dt_\Pi
&=
e^{i\gamma_x}
\int_{\Pi_d}
\exp\left(i\left[
\sum_{j=1}^m\tau_jq_j(t_j)
+
\sum_{k=m+1}^d\tau_kp_k(t_k)
+
\sum_{j=1}^d\beta_jt_j
\right]\right)
A(t)\,dt_\Pi \\
&=
e^{i\gamma_x}\mathcal I(\tau,\beta),
\end{aligned}
\]
where \(\mathcal I\) is defined in
\eqref{eq:separable-integral}, with \(a=A\).  Since
\(|x_j|\simeq |x|\) for every \(j\), its hypotheses hold with
\(R=|x|\).  The factor \(e^{i\gamma_x}\) has modulus one, while the
linear Taylor terms are precisely the arbitrary linear perturbation
allowed in \(\mathcal I(\tau,\beta)\).

If \(0\leq m\leq4\), Proposition~\ref{prop:low-flat} gives
\[
|\widehat{\chi\,d\sigma_{M_\lambda}}(x)|
\leq
C
|x|^{-(d-2)/2}
\bigl(\log(2+|x|)\bigr)^{m/4}.
\]
If \(m\geq5\), Proposition~\ref{prop:high-flat} gives
\[
|\widehat{\chi\,d\sigma_{M_\lambda}}(x)|
\leq
C|x|^{-(d-m)/2-(m-1)/3}.
\]
These are the asserted estimates.
\end{proof}

\begin{proposition}[Regular patches with vanishing sine coordinates]
\label{prop:regular-general}
The conclusions of Proposition~\ref{prop:regular-cosine} remain valid for
every regular point \(\xi^0\in M_\lambda\), without the assumption that all
\(\sin(2\pi\xi_j^0)\) are nonzero.  Here
\(m=\#\{j:\cos(2\pi\xi_j^0)=0\}\).
\end{proposition}

\begin{proof}
It is enough to consider \(|x|\geq1\).
Write
\(\theta_j=2\pi\xi_j, S=\{j:\sin\theta_j^0=0\}\).
Since \(\xi^0\) is regular, \(S\neq\{1,\dots,d\}\). Fix \(j_0\notin S\).

For \(j\notin S\), use the cosine coordinate
\(c_j=\cos\theta_j=c_j^0+t_j\).
For \(j\in S\), retain the angular coordinate
\(\theta_j=\theta_j^0+u_j, \sigma_j=\cos\theta_j^0\in\{1,-1\}\).
The surface equation becomes
\(\sum_{j\notin S}t_j + \sum_{j\in S}\sigma_j(\cos u_j-1) = 0\).
Thus a localized Fourier transform has the form
\[
\int
e^{i\Phi_x(t,u)}
A(t,u)
\delta\left(
\sum_{j\notin S}t_j
+
\sum_{j\in S}\sigma_j(\cos u_j-1)
\right)
\,dt\,du,
\]
and denote this integral by \(I_x\), where
\(\Phi_x(t,u) = \sum_{j\notin S}x_j\theta_j(c_j^0+t_j) + \sum_{j\in S}x_ju_j\),
up to an irrelevant constant.

Let
\(\nu=(-\sin\theta_1^0,\dots,-\sin\theta_d^0)\).
Suppose first that
\(\operatorname{dist}(x,\mathbb R\nu)\geq\varepsilon|x|\).
Use the surface equation to eliminate \(t_{j_0}\):
\[
t_{j_0}
=
-\sum_{\substack{j\notin S\\ j\neq j_0}}t_j
-\sum_{j\in S}\sigma_j(\cos u_j-1).
\]
Let \(y\in\mathbb R^{d-1}\) denote the remaining variables, and continue
to write \(\Phi_x\) and \(A\) for the resulting phase and amplitude. The
delta integration in \(t_{j_0}\) has unit Jacobian, so the localized
integral becomes
\(\int_{\mathbb R^{d-1}}e^{i\Phi_x(y)}A(y)\,dy\).

At \(y=0\),
\[
\nabla_y\Phi_x(0)
=
\left(
\bigl(x_j\theta_j'(c_j^0)
-x_{j_0}\theta_{j_0}'(c_{j_0}^0)\bigr)_
{\substack{j\notin S\\ j\neq j_0}},
(x_j)_{j\in S}
\right).
\]
The fixed linear map
\(x\longmapsto\nabla_y\Phi_x(0)\)
has kernel \(\mathbb R\nu\). Hence
\(|\nabla_y\Phi_x(0)| \geq c\,\operatorname{dist}(x,\mathbb R\nu) \geq c\varepsilon|x|\).
Since
\[
|\partial_y^\alpha\Phi_x(y)|
\leq
C_\alpha|x|,
\qquad
|\alpha|\geq1,
\]
shrinking the coordinate support gives
\[
|\nabla_y\Phi_x(y)|\geq c_\varepsilon|x|
\]
on \(\operatorname{supp}A\).  Lemma~\ref{lem:nonstationary}, applied with
\(E=\mathbb R^{d-1}\) and \(R=|x|\), therefore gives, for every \(N\geq1\),
\(\left| \int_{\mathbb R^{d-1}} e^{i\Phi_x(y)}A(y)\,dy \right| \leq C_{N,\varepsilon}|x|^{-N}\).
This is stronger than the asserted estimates.

It remains to consider
\(\operatorname{dist}(x,\mathbb R\nu)<\varepsilon|x|\).
Choose \(\rho\) such that
\(|x-\rho\nu| = \operatorname{dist}(x,\mathbb R\nu)\).
Then
\(|\rho|\simeq |x|=:R\).
For \(j\notin S\), one has
\(x_j\theta_j'(c_j^0)=\rho+O(\varepsilon R)\).

For a product amplitude, the Fourier representation of the delta mass gives
\(\frac1{2\pi} \int_{\mathbb R} \prod_{j\notin S}D_j(s) \prod_{j\in S}P_j(s)\,ds\),
where
\(D_j(s) = \int e^{i(x_j\theta_j(c_j^0+t)+st)} \psi_j(t)\,dt\),
and
\(P_j(s) = \int e^{i(x_ju+s\sigma_j(\cos u-1))} \psi_j(u)\,du\).

Choose \(\eta>0\) sufficiently small so that
\(|s+\rho|\leq\eta R \quad\Longrightarrow\quad |s|\simeq R\).
By taking \(\varepsilon\) and the support of \(\psi_{j_0}\) sufficiently
small, we may also assume that
\[
\left|x_{j_0}\theta_{j_0}'(c_{j_0}^0)-\rho\right|
+
|x_{j_0}|
\sup_{t\in\operatorname{supp}\psi_{j_0}}
\left|
\theta_{j_0}'(c_{j_0}^0+t)-\theta_{j_0}'(c_{j_0}^0)
\right|
\leq
\frac{\eta}{2}R.
\]
Set
\(J=\{s:|s+\rho|\leq\eta R\}\).

For \(s\notin J\), let
\(\phi_s(t) = x_{j_0}\theta_{j_0}(c_{j_0}^0+t)+st\).
Then
\[
\begin{aligned}
|\phi_s'(t)|
&=
\left|
s+x_{j_0}\theta_{j_0}'(c_{j_0}^0+t)
\right|
\\
&\geq
|s+\rho|-\frac{\eta}{2}R
\\
&\geq
\frac12|s+\rho|
\end{aligned}
\]
on \(\operatorname{supp}\psi_{j_0}\). Moreover,
\(|\phi_s^{(r)}(t)| \leq C_{r,\eta}|s+\rho|, r\geq1\).
Indeed, the estimate for \(r=1\) also follows from the preceding
decomposition, while for \(r\geq2\),
\(|\phi_s^{(r)}(t)| \leq C_r|x_{j_0}| \leq C_rR \leq C_{r,\eta}|s+\rho|\).

Lemma~\ref{lem:nonstationary}, applied on \(\mathbb R\) with scale
\(|s+\rho|\),
and with \(N+1\) in place of \(N\), gives, for every \(N\geq1\),
\(|D_{j_0}(s)| \leq C_{N,\eta}|s+\rho|^{-N-1}, s\notin J\).
Using the trivial bounds for all the other one-dimensional factors, we
obtain
\[
\begin{aligned}
&
\int_{\mathbb R\setminus J}
\prod_{j\notin S}|D_j(s)|
\prod_{j\in S}|P_j(s)|\,ds
\\
&\qquad\leq
C_{N,\eta}
\int_{|s+\rho|>\eta R}
|s+\rho|^{-N-1}\,ds
\\
&\qquad\leq
C_{N,\eta}R^{-N}.
\end{aligned}
\]
Thus the contribution of \(\mathbb R\setminus J\) has arbitrary decay in
\(R\).

On \(J\), one has
\(|s|\simeq R\).
For \(j\in S\), the second derivative of the phase defining \(P_j\) is
\(-s\sigma_j\cos u\).
After shrinking the support in \(u\),
\(|-s\sigma_j\cos u|\geq cR\).
Lemma~\ref{lem:vandercorput} therefore gives
\(\|P_j\|_{L^\infty(J)} \leq CR^{-1/2}\).
Since \(|J|\lesssim R\), it follows that
\[
\|P_j\|_{L^p(J)}
\leq
C_pR^{-1/2+1/p},
\qquad
2\leq p\leq\infty.
\]

For \(j\notin S\), the factors \(D_j\) are the cubic or quadratic
one-dimensional factors controlled by Lemmas~\ref{lem:cubic-1d} and
\ref{lem:quadratic-1d}.  Moreover, \(\cos\theta_j^0=0\) can only occur
when \(j\notin S\).  Consequently, there are exactly \(m\) cubic
factors and \(d-m\) quadratic factors, with the factors \(P_j\),
\(j\in S\), counted among the quadratic ones.

The same H\"older arguments as in Propositions~\ref{prop:low-flat} and
\ref{prop:high-flat} now give
\[
|I_x|
\leq
C
R^{-(d-2)/2}
\bigl(\log(2+R)\bigr)^{m/4},
\qquad
0\leq m\leq4,
\]
and
\(|I_x| \leq C R^{-(d-m)/2-(m-1)/3}, m\geq5\).

Finally, we pass from product amplitudes to a general amplitude by the
periodization argument used in Lemma~\ref{lem:reduction}.
Writing \(y=(t,u)\), choose a fixed cube containing
\(\operatorname{supp}A\) and fixed product cutoffs
\(\vartheta_k\) which are equal to \(1\) on that support.  Periodizing
\(A\) on this cube gives
\[A(y)=\sum_{\nu\in\mathbb Z^d}c_\nu\prod_{k=1}^d\vartheta_k(y_k)e^{i\kappa\nu_ky_k},
\qquad
|c_\nu|\leq C_M(1+|\nu|)^{-M}\]
for every \(M\).

Choose \(\delta>0\) sufficiently small relative to
\(\varepsilon\) and \(\eta\).  If \(|\nu|\leq\delta R\), the added
phase \(\kappa\nu\cdot y\) perturbs the relevant phase derivatives by
\(O(\delta R)\).  Hence Lemma~\ref{lem:nonstationary} remains applicable
in the non-normal case.  In the normal case, the factors \(D_j(s)\) are
translated in \(s\) by \(O(\delta R)\), while the added terms in the
phases of the \(P_j(s)\) are linear in \(u\).  Thus the definition of
\(J\), the estimate on \(\mathbb R\setminus J\), and the estimates on
\(J\) remain valid, with uniform constants.  Summing these modes is
permitted since \(\sum_\nu|c_\nu|<\infty\).

For \(|\nu|>\delta R\), the corresponding compactly supported surface
integral is bounded trivially by a constant independent of \(x\) and
\(\nu\).  Therefore
\[
\sum_{|\nu|>\delta R}|c_\nu|
\lesssim_M
\sum_{|\nu|>\delta R}(1+|\nu|)^{-M}
\lesssim_M R^{d-M}.
\]
Taking \(M>N+d\) shows that the contribution of these modes is
\(O_N(R^{-N})\) for every \(N\).  This proves the claim.

\end{proof}

\begin{proposition}[Threshold critical patches]
\label{prop:threshold-decay}
Let \(d\geq4\), and let \(\xi^0\in M_\lambda\) satisfy
\(\sin(2\pi\xi_j^0)=0, j=1,\dots,d\).
Let \(\chi\) be supported sufficiently close to \(\xi^0\). Then, on the
regular part of the local level set,
\(|\widehat{\chi\,d\sigma_{M_\lambda}}(x)| \leq C(1+|x|)^{-(d-2)/2}\).
At the bottom and top thresholds \(\lambda=0\) and \(\lambda=4d\), the local
level set has no \((d-1)\)-dimensional regular part.
\end{proposition}

\begin{proof}
Set
\[
\theta_j=2\pi\xi_j,
\qquad
\theta_j=\theta_j^0+u_j,
\qquad
\sigma_j=\cos\theta_j^0\in\{1,-1\}.
\]
Suppose that exactly \(q\) of the coordinates \(\theta_j^0\) are equal to
\(\pi\). Then
\(\lambda=4q\).
Moreover, since
\(\cos(\theta_j^0+u_j)=\sigma_j\cos u_j\),
one has the exact identity, in angular coordinates,
\(h_0\left(\frac{\theta^0+u}{2\pi}\right)-4q = 2\sum_{j=1}^d\sigma_j(1-\cos u_j)\).

If \(q=0\), then all \(\sigma_j=1\), and the local level-set equation is
\(\sum_{j=1}^d(1-\cos u_j)=0\).
Since every term is nonnegative, this implies \(u=0\). If \(q=d\), the
same conclusion follows after multiplying the equation by \(-1\).
Thus, at \(\lambda=0\) and \(\lambda=4d\), there is no local
\((d-1)\)-dimensional regular part.

Assume from now on that
\(1\leq q\leq d-1\).
Set
\(F(u)=\sum_{j=1}^d\sigma_j(1-\cos u_j)\).
Up to an irrelevant constant phase and a fixed normalization constant, the
localized surface Fourier transform is
\[
I_x
=
\int_{\mathbb R^d}
e^{ix\cdot u}
a(u)
|\nabla F(u)|
\delta(F(u))
\,du,
\]
where \(a\) is smooth and supported sufficiently close to the origin. We
may assume that
\(|\cos u_j|\geq\frac12\)
on the support.

It is enough to consider \(|x|\geq1\). Decompose the punctured support into
dyadic regions
\(|u|\simeq 2^{-k}\).
On such a region, make the change of variables
\(u=2^{-k}v\)
and set
\[
F_k(v)
=
2^{2k}F(2^{-k}v)
=
\sum_{j=1}^d
\sigma_j2^{2k}
\bigl(1-\cos(2^{-k}v_j)\bigr).
\]
Then
\(\partial_{v_j}F_k(v) = \sigma_j2^k\sin(2^{-k}v_j)\),
and
\(\partial_{v_j}^2F_k(v) = \sigma_j\cos(2^{-k}v_j)\).
Also,
\(F(2^{-k}v)=2^{-2k}F_k(v)\),
so that
\(\delta(F(2^{-k}v)) = 2^{2k}\delta(F_k(v))\).
Furthermore,
\[
|\nabla F(2^{-k}v)|
=
2^{-k}
\left(
\sum_{j=1}^d
\bigl(2^k\sin(2^{-k}v_j)\bigr)^2
\right)^{1/2}.
\]
On the fixed annulus \(|v|\simeq1\), the expression in parentheses is
smooth and bounded above and below uniformly in \(k\). Hence the contribution
of the dyadic region has the form
\[
\frac{2^{-k(d-1)}}{2\pi}
\int_{\mathbb R}
\int_{\mathbb R^d}
e^{i(2^{-k}x\cdot v+sF_k(v))}
A_k(v)
\,dv\,ds,
\]
where the functions \(A_k\) are supported in a fixed annulus and are bounded
uniformly in \(C^\infty\).

Suppose first that
\(2^{-k}|x|\geq1\).
Choose \(j\) such that
\(|x_j|\geq\frac{|x|}{\sqrt d}\).
The derivative of the phase in \(v_j\) is
\(2^{-k}x_j + s\sigma_j2^k\sin(2^{-k}v_j)\).
Since \(v\) ranges over a fixed compact set,
\(2^k|\sin(2^{-k}v_j)| \leq |v_j| \leq C\).
Consequently, if
\[
|s|\leq c\,2^{-k}|x|
\]
and \(c>0\) is sufficiently small, then
\[
\left|
2^{-k}x_j+s\sigma_j2^k\sin(2^{-k}v_j)
\right|
\geq
\frac{2^{-k}|x|}{2\sqrt d}.
\]
Lemma~\ref{lem:nonstationary}, applied with scale
\(2^{-k}|x|\), therefore gives, for every \(N\),
\[
\left|
\int_{\mathbb R^d}
e^{i(2^{-k}x\cdot v+sF_k(v))}
A_k(v)\,dv
\right|
\leq
C_N(2^{-k}|x|)^{-N}.
\]
After integration over
\(|s|\leq c\,2^{-k}|x|\),
this part is bounded by
\(C_N(2^{-k}|x|)^{1-N}\).
Choosing \(N\) sufficiently large gives
\(C_N(2^{-k}|x|)^{1-N} \leq C(2^{-k}|x|)^{-(d-2)/2}\).

It remains to consider
\(|s|>c\,2^{-k}|x|\).
For every \(j\),
\[
\left|
\partial_{v_j}^2
\bigl(
2^{-k}x\cdot v+sF_k(v)
\bigr)
\right|
=
|s|\,|\cos(2^{-k}v_j)|
\geq
\frac{|s|}{2}.
\]
Choose fixed functions
\(\vartheta_j\in C_c^\infty(\mathbb R)\) such that
\(\prod_{j=1}^d\vartheta_j(v_j)=1\) on the supports of all \(A_k\).
Periodizing on a fixed box containing these supports, we may write
\[
A_k(v)
=
\sum_{\nu\in\mathbb Z^d}
c_{k,\nu}
\prod_{j=1}^d
\vartheta_j(v_j)e^{i\kappa\nu_jv_j},
\qquad
|c_{k,\nu}|
\leq
C_M(1+|\nu|)^{-M}
\]
for every \(M\), uniformly in \(k\), where \(\kappa>0\) depends only on
the fixed box.  For each Fourier mode, the corresponding integral
factors as
\[
\prod_{j=1}^d
\int_{\mathbb R}
\exp\left(
i\left[
(2^{-k}x_j+\kappa\nu_j)v_j
+
s\sigma_j2^{2k}
\bigl(1-\cos(2^{-k}v_j)\bigr)
\right]
\right)
\vartheta_j(v_j)\,dv_j.
\]
Thus the Fourier mode only replaces the linear coefficient
\(2^{-k}x_j\) by \(2^{-k}x_j+\kappa\nu_j\).  In particular, it does not
change the second derivative
\[
s\sigma_j\cos(2^{-k}v_j),
\]
whose absolute value is at least \(|s|/2\) on the support.  The
second-order estimate in Lemma~\ref{lem:vandercorput} therefore gives,
uniformly in \(\nu\),
\[
\left|
\prod_{j=1}^d
\int_{\mathbb R}\cdots\,dv_j
\right|
\leq
C(1+|s|)^{-d/2}.
\]
Since
\[
\sup_k\sum_{\nu\in\mathbb Z^d}|c_{k,\nu}|<\infty,
\]
summing the Fourier modes gives
\[
\left|
\int_{\mathbb R^d}
e^{i(2^{-k}x\cdot v+sF_k(v))}
A_k(v)\,dv
\right|
\leq
C(1+|s|)^{-d/2}.
\]
Hence
\[
\begin{aligned}
\int_{|s|>c\,2^{-k}|x|}
(1+|s|)^{-d/2}\,ds
&\leq
C(2^{-k}|x|)^{1-d/2}
\\
&=
C(2^{-k}|x|)^{-(d-2)/2}.
\end{aligned}
\]

Combining the two ranges of \(s\), the contribution of a dyadic region
satisfying \(2^{-k}|x|\geq1\) is bounded by
\[
\begin{aligned}
C\,2^{-k(d-1)}
(2^{-k}|x|)^{-(d-2)/2}
&=
C|x|^{-(d-2)/2}
2^{-k\left(d-1-(d-2)/2\right)}
\\
&=
C|x|^{-(d-2)/2}2^{-kd/2}.
\end{aligned}
\]
Therefore
\[
\sum_{2^{-k}|x|\geq1}
C|x|^{-(d-2)/2}2^{-kd/2}
\leq
C|x|^{-(d-2)/2}.
\]

Finally, if
\(2^{-k}|x|<1\),
we use the trivial estimate. Since
\[
|\nabla F_k(v)|^2
=
\sum_{j=1}^d
\bigl(2^k\sin(2^{-k}v_j)\bigr)^2
\simeq
|v|^2
\simeq1
\]
on the fixed annulus, the rescaled level sets have uniformly bounded surface
measure. Thus the contribution of such a dyadic region is at most
\(C2^{-k(d-1)}\).
Summing gives
\(\sum_{2^{-k}|x|<1}2^{-k(d-1)} \leq C|x|^{-(d-1)}\).

Combining the two dyadic ranges yields
\(|I_x| \leq C|x|^{-(d-2)/2}\).
This proves
\(|\widehat{\chi\,d\sigma_{M_\lambda}}(x)| \leq C(1+|x|)^{-(d-2)/2}\).
\end{proof}

\begin{proposition}[Global and away-from-threshold decay]
\label{prop:global-decay}
Let \(d\geq4\), and let \(\chi\) be supported sufficiently close to a point of
\(M_\lambda\).

If \(d=4\), then
\(|\widehat{\chi\,d\sigma_{M_\lambda}}(x)| \leq C(1+|x|)^{-1}\log(2+|x|)\).

If \(d\geq5\), then
\(|\widehat{\chi\,d\sigma_{M_\lambda}}(x)| \leq C(1+|x|)^{-(d-1)/3}\).

If \(d\geq6\) is even and
\(\operatorname{dist} \left( \lambda,\{4j:0\leq j\leq d\} \right) \geq\varepsilon>0\),
then
\(|\widehat{\chi\,d\sigma_{M_\lambda}}(x)| \leq C_\varepsilon (1+|x|)^{-(2d-1)/6}\).
\end{proposition}

\begin{proof}
For fixed \(\lambda\), choose a finite smooth partition of unity
\(\{\eta_\nu\}\) on a neighborhood of \(M_\lambda\), subordinate to
sufficiently small neighborhoods on which
Proposition~\ref{prop:regular-general} or
Proposition~\ref{prop:threshold-decay} applies.  Since
\[
\widehat{\chi\,d\sigma_{M_\lambda}}(x)
=
\sum_\nu
\widehat{\chi\eta_\nu\,d\sigma_{M_\lambda}}(x),
\]
it suffices to estimate each localized factor \(\chi\eta_\nu\).
In the remainder of the proof, we relabel such a factor as \(\chi\)
and denote the center of its supporting patch by \(\xi^0\).  The
small-support conditions ensure that the relevant coordinate branches
are defined throughout the patch and that all required nonvanishing
derivative bounds persist on its support.

We first consider a regular patch, and set
\[
m=\#\{j:\cos(2\pi\xi_j^0)=0\}.
\]
Proposition~\ref{prop:regular-general} applies whether or not some of
the quantities \(\sin(2\pi\xi_j^0)\) vanish.

Suppose first that
\(0\leq m\leq4\).
Then
\[
|\widehat{\chi\,d\sigma_{M_\lambda}}(x)|
\leq
C
(1+|x|)^{-(d-2)/2}
\bigl(\log(2+|x|)\bigr)^{m/4}.
\]
If \(d=4\), then
\(\frac{d-2}{2}=1\),
and, since \(m\leq4\),
\(\bigl(\log(2+|x|)\bigr)^{m/4} \leq C\log(2+|x|)\).
Hence
\(|\widehat{\chi\,d\sigma_{M_\lambda}}(x)| \leq C(1+|x|)^{-1}\log(2+|x|)\).

Now let \(d\geq5\). We have
\(\frac{d-2}{2}-\frac{d-1}{3} = \frac{d-4}{6}>0\).
Choose
\(0<\delta<\frac{d-4}{6}\).
Since
\[
\bigl(\log(2+R)\bigr)^{m/4}
\leq
C_\delta(1+R)^\delta,
\qquad
R\geq0,
\]
it follows that
\[
\begin{aligned}
&(1+R)^{-(d-2)/2}
\bigl(\log(2+R)\bigr)^{m/4}
\\
&\qquad\leq
C_\delta
(1+R)^{-(d-2)/2+\delta}
\leq
C_\delta
(1+R)^{-(d-1)/3}.
\end{aligned}
\]
Thus all regular patches with \(0\leq m\leq4\) satisfy the asserted global
decay when \(d\geq5\).

Suppose next that
\[
m\geq5.
\]
Proposition~\ref{prop:regular-general} gives
\[
|\widehat{\chi\,d\sigma_{M_\lambda}}(x)|
\leq
C(1+|x|)^{-k_m},
\qquad
k_m=\frac{d-m}{2}+\frac{m-1}{3}.
\]
A direct calculation gives
\(k_m-\frac{d-1}{3} = \frac{d-m}{6}\geq0\).
Therefore
\(|\widehat{\chi\,d\sigma_{M_\lambda}}(x)| \leq C(1+|x|)^{-(d-1)/3}\).
This proves the desired estimate on every regular patch when \(d\geq5\).

We now consider a threshold critical patch.
Proposition~\ref{prop:threshold-decay} gives
\[
|\widehat{\chi\,d\sigma_{M_\lambda}}(x)|
\leq
C(1+|x|)^{-(d-2)/2}.
\]
Moreover,
\(\frac{d-2}{2}-\frac{d-1}{3} = \frac{d-4}{6}\geq0\).
Thus threshold critical patches satisfy
\(|\widehat{\chi\,d\sigma_{M_\lambda}}(x)| \leq C(1+|x|)^{-(d-1)/3}\)
for every \(d\geq4\). In particular, when \(d=4\), this is stronger than the
regular endpoint estimate with its logarithmic loss.

It remains to prove the even-dimensional away-from-threshold estimate. Assume
that \(d\geq6\) is even and
\(\operatorname{dist} \left( \lambda,\{4j:0\leq j\leq d\} \right) \geq\varepsilon\).
There are no threshold critical patches.

For \(0\leq m\leq4\), Proposition~\ref{prop:regular-general} gives
the low-flat exponent, which we compare with the desired away exponent:
\[
\frac{d-2}{2}-\frac{2d-1}{6}
=
\frac{d-5}{6}
>0.
\]
Choose
\(0<\delta<\frac{d-5}{6}\).
Absorbing the logarithmic factor as above gives
\[
\begin{aligned}
&(1+R)^{-(d-2)/2}
\bigl(\log(2+R)\bigr)^{m/4}
\\
&\qquad\leq
C_\delta
(1+R)^{-(d-2)/2+\delta}
\leq
C_\delta
(1+R)^{-(2d-1)/6}.
\end{aligned}
\]
Hence the desired away estimate holds for \(0\leq m\leq4\).

Finally, suppose that \(m\geq5\). We claim that
\(m\leq d-1\).
Indeed, if \(m=d\), then
\(\cos(2\pi\xi_j^0)=0, j=1,\dots,d\).
Consequently,
\(\sum_{j=1}^d\cos(2\pi\xi_j^0)=0\),
and hence
\(\lambda=2d\).
Since \(d\) is even,
\(2d=4\cdot\frac d2\)
is a threshold energy, contrary to the assumption that \(\lambda\) is away
from thresholds. Therefore \(m\leq d-1\).

The function
\[
k_m=\frac{d-m}{2}+\frac{m-1}{3}
\]
is decreasing in \(m\).  Hence
\[
k_m
\geq
k_{d-1}
=
\frac12+\frac{d-2}{3}
=
\frac{2d-1}{6}.
\]
Proposition~\ref{prop:regular-general} therefore gives
\[
|\widehat{\chi\,d\sigma_{M_\lambda}}(x)|
\leq
C_\varepsilon(1+|x|)^{-(2d-1)/6}.
\]

Since the partition of unity is finite, summing the estimates for the
localized factors \(\chi\eta_\nu\) proves the corresponding estimate for
each fixed energy.

The constants required in the resolvent estimates may also be chosen
uniformly.  Choose fixed small neighborhoods of the finite set
\(\operatorname{Cr}(h_0)\).  Their complement is compact and regular,
and hence admits a finite cover by the coordinate patches used above,
including the affine-hyperplane versions.  On this finite cover, all
coordinate Jacobians, amplitude seminorms, and nonvanishing derivative
lower bounds entering the one-dimensional estimates are uniform.

Finally, if
\[
\operatorname{dist}
\left(
\lambda,\{0,4,\dots,4d\}
\right)
\geq\varepsilon,
\]
then
\[
M_\lambda
\subset
\left\{
\xi\in\mathbb T^d:
\operatorname{dist}
\left(
h_0(\xi),\{0,4,\dots,4d\}
\right)
\geq\varepsilon
\right\}.
\]
The set on the right is compact and disjoint from
\(\operatorname{Cr}(h_0)\).  A finite regular cover of this set therefore
gives the asserted \(C_\varepsilon\) uniformity. 
\end{proof}

\begin{proof}[Proof of Theorem~\ref{thm:local-decay}]
Proposition~\ref{prop:regular-cosine} proves the result when none of the
sine coordinates vanish.  Proposition~\ref{prop:regular-general} removes
that restriction at every regular point, and
Proposition~\ref{prop:threshold-decay} treats the critical points.  These
three propositions give all the assertions of the theorem.
\end{proof}

\begin{proof}[Proof of Theorem~\ref{thm:resolvent}]
At every \(\xi^0\in\operatorname{Cr}(h_0)\),
\[
\nabla^2h_0(\xi^0)
=
8\pi^2\operatorname{diag}
\bigl(
\cos(2\pi\xi_1^0),\dots,\cos(2\pi\xi_d^0)
\bigr),
\]
so all critical points are nondegenerate.  Choose disjoint small
neighborhoods of these points and a smooth partition of unity consisting
of cutoffs supported in these neighborhoods and cutoffs supported in the
regular region.  Proposition~\ref{prop:critical-resolvent} controls every
critical cutoff for \(1\leq r\leq d\).

On the regular cutoffs, Proposition~\ref{prop:global-decay} and
Proposition~\ref{prop:fourier-to-resolvent} give, for \(d\geq5\),
\(1\leq r\leq 2+\frac{2(d-1)}3 = \frac{2(d+2)}3\).
This upper endpoint does not exceed \(d\), so the same range is available
on the critical cutoffs.  For \(2\leq r\leq2(d+2)/3\),
Lemma~\ref{lem:weighted-unweighted}, with
\(p=\frac{2r}{r+2}\),
gives
\(1\leq p\leq\frac{2(d+2)}{d+5}\).
The weighted estimates for \(1\leq r<2\) follow from the case \(r=2\)
and the inclusion \(\ell^r\subset\ell^2\).

When \(d=4\), Propositions~\ref{prop:global-decay} and
\ref{prop:fourier-to-resolvent} give every \(r<4\), and the range
\(2\leq r<4\) is equivalent to every \(1\leq p<4/3\), uniformly
in the spectral parameter.  Again, \(1\leq r<2\) follows from \(r=2\).
Proposition~\ref{prop:log-endpoint} gives the stated square-root
logarithmic loss for \(r=4\) and \(p=4/3\).  The critical cutoffs are
uniformly bounded at \(r=4\), so they obey the same weaker logarithmic
bound.

It remains to consider the improved even-dimensional estimate away from
the thresholds.  Let \(d\geq6\) be even and
\(z\in D_\delta\setminus\mathbb R\).  Choose a smooth cutoff \(\eta_\delta\) such that
\[
\eta_\delta(\xi)=1
\quad\text{when}\quad
\operatorname{dist}
\left(
h_0(\xi),\{0,4,\dots,4d\}
\right)
\leq\frac{\delta}{8},
\]
and
\[
\operatorname{supp}\eta_\delta
\subset
\left\{
\xi:
\operatorname{dist}
\left(
h_0(\xi),\{0,4,\dots,4d\}
\right)
<\frac{\delta}{4}
\right\}.
\]
If \(z\in D_\delta\setminus\mathbb R\) and
\(\xi\in\operatorname{supp}\eta_\delta\), then, for some
\(j\in\{0,\dots,d\}\),
\[
|h_0(\xi)-z|
\geq
|z-4j|-|h_0(\xi)-4j|
\geq
\frac{3\delta}{4}.
\]
Consequently, the multipliers
\[
m_z(\xi)
=
\frac{\eta_\delta(\xi)}{h_0(\xi)-z}
\]
are uniformly bounded in \(C^\infty(\mathbb T^d)\), with bounds
depending only on \(\delta\).  Repeated integration by parts in their
Fourier coefficients gives
\[
|\widehat m_z(n)|
\leq
C_{N,\delta}(1+|n|)^{-N},
\qquad n\in\mathbb Z^d,
\]
for every \(N\), uniformly in \(z\in D_\delta\setminus\mathbb R\).
Thus this part of the resolvent has a uniformly rapidly decaying
convolution kernel and satisfies all the required estimates.

On the support of \(1-\eta_\delta\), every level surface meeting the
support has energy at distance at least \(\delta/8\) from the
thresholds.  Proposition~\ref{prop:global-decay} therefore gives decay
exponent \(\frac{2d-1}{6}\).
Proposition~\ref{prop:fourier-to-resolvent} therefore gives
\(1\leq r\leq 2+\frac{2d-1}{3} = \frac{2d+5}{3}\),
and the subrange \(2\leq r\leq(2d+5)/3\) is equivalent to
\(1\leq p\leq\frac{2(2d+5)}{2d+11}\).
For odd \(d\geq5\), the global estimate already gives the stated
away-from-threshold range.  This completes the proof.
\end{proof}
\section{Necessary conditions and sharpness}\label{sec:sharp}

Fix $\psi\in C_c^\infty((-1,1))$ with $0\leq\psi\leq1$ and $\psi(t)=1$ for $|t|\leq1/2$.

\begin{lemma}
\label{lem:anisotropic-box}
Let $\xi^0\in(-1/2,1/2)^d$ and $L\in GL_d(\mathbb R)$. For $0<\rho_1,\dots,\rho_d\leq1$, define
\[
\widehat f(\xi) = \prod_{j=1}^d \psi\!\left( \frac{(L(\xi-\xi^0))_j}{\rho_j} \right).
\]
Assume that the support of this function is contained in the interior of $(-1/2,1/2)^d$. Put $V=\rho_1\cdots\rho_d$. Then for every $1\leq p<\infty$,
\(\|f\|_{\ell^p(\mathbb Z^d)} \leq C_{p,L,\psi} V^{1-1/p}\).
\end{lemma}

\begin{proof}
Set $y=L(\xi-\xi^0)$. Then $\xi=\xi^0+L^{-1}y$ and $d\xi=|\det L|^{-1}\,dy$. Hence
\[f(n) = |\det L|^{-1}e^{2\pi i n\cdot\xi^0} \int_{\mathbb R^d} e^{2\pi i(L^{-T}n)\cdot y} \prod_{j=1}^d \psi\!\left(\frac{y_j}{\rho_j}\right)\,dy.\]
Let $y_j=\rho_j t_j$. Then
\[f(n) = |\det L|^{-1}e^{2\pi i n\cdot\xi^0} V \prod_{j=1}^d \widehat\psi_{\mathbb R}\!\left(\rho_j(L^{-T}n)_j\right),\]
where $\widehat\psi_{\mathbb R}(s)=\int_{\mathbb R}e^{2\pi ist}\psi(t)\,dt$. For every $M>0$, we have $|\widehat\psi_{\mathbb R}(s)|\leq C_M(1+|s|)^{-M}$, so
\(|f(n)| \leq C_M V \prod_{j=1}^d \left(1+\rho_j|(L^{-T}n)_j|\right)^{-M}\).

To estimate the $\ell^p$-norm of $f$, we compare the sum over $\mathbb{Z}^d$ with an integral over $\mathbb{R}^d$. For any $x \in n + [0,1]^d$, we have $|x - n| \leq \sqrt{d}$. Therefore,
\[|(L^{-T}n)_j| \geq |(L^{-T}x)_j| - |(L^{-T}(n-x))_j| \geq |(L^{-T}x)_j| - C_L,\]
where $C_L > 0$ is a constant depending only on $L$. Since $0 < \rho_j \leq 1$, it follows that
\(1 + \rho_j|(L^{-T}n)_j| \geq c_L \left(1 + \rho_j|(L^{-T}x)_j|\right)\)
for some constant $c_L > 0$ depending only on $L$. 

Taking $M$ sufficiently large such that $Mp > 1$, the above inequality implies
\[\prod_{j=1}^d \left(1+\rho_j|(L^{-T}n)_j|\right)^{-Mp} \leq C \int_{n+[0,1]^d} \prod_{j=1}^d \left(1+\rho_j|(L^{-T}x)_j|\right)^{-Mp} \,dx.\]
Summing over all $n \in \mathbb{Z}^d$ yields
\[\sum_{n\in\mathbb Z^d} \prod_{j=1}^d \left(1+\rho_j|(L^{-T}n)_j|\right)^{-Mp} \leq C \int_{\mathbb R^d} \prod_{j=1}^d \left(1+\rho_j|(L^{-T}x)_j|\right)^{-Mp} \,dx.\]
We evaluate the integral using the change of variables $u = L^{-T}x$, which gives $dx = |\det L| \,du$:
\[\int_{\mathbb R^d} \prod_{j=1}^d \left(1+\rho_j|(L^{-T}x)_j|\right)^{-Mp} \,dx = |\det L| \prod_{j=1}^d \int_{\mathbb R} (1+\rho_j|u_j|)^{-Mp} \,du_j 
= \widetilde{C} \prod_{j=1}^d \rho_j^{-1} = \widetilde{C} V^{-1}.\]
Therefore,
\(\|f\|_{\ell^p}^p \leq C V^p V^{-1} = C V^{p-1}\),
and hence $\|f\|_{\ell^p} \leq C V^{1-1/p}$.
\end{proof}

\begin{lemma}
\label{lem:resolvent-obstruction}
Let $1\leq p\leq2$ and $0<\varepsilon<\varepsilon_0$. Suppose $f_\varepsilon$ satisfies Lemma~\ref{lem:anisotropic-box} with volume product $V_\varepsilon\simeq\varepsilon^a$. Let
\[
E_\varepsilon = \left\{ \xi : |(L(\xi-\xi^0))_j| \leq \frac{\rho_j}{2},\ j=1,\dots,d \right\}.
\]
Assume that $|h_0(\xi)-\lambda|\leq C\varepsilon$ for $\xi\in E_\varepsilon$. If the uniform resolvent bound
\(\sup_{0<\varepsilon<\varepsilon_0} \| (H_0-\lambda-i\varepsilon)^{-1} \|_{\ell^p\to\ell^{p'}} <\infty\)
holds, then $p\leq \frac{2a}{a+1}$.
\end{lemma}

\begin{proof}
On $E_\varepsilon$, we have $\widehat f_\varepsilon(\xi)=1$, and $|E_\varepsilon|=|\det L|^{-1}V_\varepsilon$. By Plancherel,
\begin{align*}
&\operatorname{Im}
\left\langle
(H_0-\lambda-i\varepsilon)^{-1}f_\varepsilon,f_\varepsilon
\right\rangle \\
&\quad=
\int_{\mathbb T^d}
\frac{\varepsilon|\widehat f_\varepsilon(\xi)|^2}
{(h_0(\xi)-\lambda)^2+\varepsilon^2}\,d\xi \\
&\quad\geq
\int_{E_\varepsilon}
\frac{\varepsilon}{(h_0(\xi)-\lambda)^2+\varepsilon^2}\,d\xi
\geq c\varepsilon^{-1}|E_\varepsilon|
\geq c\varepsilon^{-1}V_\varepsilon.
\end{align*}
On the other hand,
\[
\left| \left\langle (H_0-\lambda-i\varepsilon)^{-1}f_\varepsilon, f_\varepsilon \right\rangle \right|
\le \| (H_0-\lambda-i\varepsilon)^{-1}f_\varepsilon \|_{\ell^{p'}} \|f_\varepsilon\|_{\ell^p}
\le C\|f_\varepsilon\|_{\ell^p}^2
\le C V_\varepsilon^{2-2/p}.
\]
Thus $\varepsilon^{-1}V_\varepsilon \leq C V_\varepsilon^{2-2/p}$. Substituting $V_\varepsilon\simeq\varepsilon^a$ gives $\varepsilon^{a-1} \leq C\varepsilon^{2a(1-1/p)}$. Letting $\varepsilon\to0$ yields $a-1 \ge 2a(1-1/p)$, so $\frac{2a}{p}\ge a+1$, hence $p\leq \frac{2a}{a+1}$.
\end{proof}

\subsection{Global obstruction}
Set $\lambda=2d$ and $\xi^0=(1/4,\dots,1/4)$. Define
\[\widehat f_\varepsilon(\xi) = \prod_{j=1}^{d-1} \psi\!\left( \frac{\xi_j-\frac14}{\varepsilon^{1/3}} \right) \psi\!\left( \frac{ \sum_{j=1}^d(\xi_j-\frac14) }{\varepsilon} \right).\]
For small $\varepsilon$ the support lies in the interior of $(-1/2,1/2)^d$. The linear transformation is
\[L(\xi-\xi^0) = \left( \xi_1-\frac14,\dots, \xi_{d-1}-\frac14,\ \sum_{j=1}^d(\xi_j-\frac14) \right),\]
with $\det L=1$. The scales are $\rho_1=\cdots=\rho_{d-1}=\varepsilon^{1/3}$ and $\rho_d=\varepsilon$, so
\(V_\varepsilon = (\varepsilon^{1/3})^{d-1}\varepsilon = \varepsilon^{(d+2)/3}\).
On the support, $|\xi_j-1/4|\leq\varepsilon^{1/3}$ for $j=1,\dots,d-1$, and $|\sum_{j=1}^d(\xi_j-1/4)|\leq\varepsilon$. Hence
\(|\xi_d-1/4| \leq \varepsilon + \sum_{j=1}^{d-1}|\xi_j-1/4| \leq C_d\varepsilon^{1/3}\).
Since $\cos(2\pi\xi_j)=-\sin(2\pi(\xi_j-1/4))=-2\pi(\xi_j-1/4)+O(|\xi_j-1/4|^3)$, we get
\[h_0(\xi)-2d = 4\pi\sum_{j=1}^d(\xi_j-1/4) + O\!\left(\sum_{j=1}^d|\xi_j-1/4|^3\right).\]
On the support, $\sum_{j=1}^d|\xi_j-1/4|^3 \leq C_d\varepsilon$, so $|h_0(\xi)-2d|\leq C_d\varepsilon$. By Lemma~\ref{lem:anisotropic-box},
\(\|f_\varepsilon\|_{\ell^p} \leq C_p \varepsilon^{\frac{d+2}{3}(1-1/p)}\).
Applying Lemma~\ref{lem:resolvent-obstruction} gives $p \leq \frac{2(d+2)}{d+5}$.

\subsection{Away from thresholds in odd dimensions}
Let $d\ge5$ be odd. Take $\lambda=2d$ and $\xi^0=(1/4,\dots,1/4)$, and define $\widehat f_\varepsilon$ as in the global case. The threshold set is $\{0,4,\dots,4d\}=\{4j:0\leq j\leq d\}$. Since $d$ is odd, $2d\equiv2\pmod4$, hence $2d\notin\{0,4,\dots,4d\}$ and $\operatorname{dist}(2d,\{0,4,\dots,4d\})=2$. The same calculation gives $V_\varepsilon=\varepsilon^{(d+2)/3}$ and $|h_0(\xi)-2d|\leq C_d\varepsilon$ on the corresponding box, and $\|f_\varepsilon\|_{\ell^p} \leq C_p \varepsilon^{\frac{d+2}{3}(1-1/p)}$. If, for some fixed \(0<\delta<1\),
\[
\sup_{z\in D_\delta\setminus\mathbb R}
\|(H_0-z)^{-1}\|_{\ell^p\to\ell^{p'}}
<\infty,
\]
then taking \(z_\varepsilon=2d+i\varepsilon\) and applying
Lemma~\ref{lem:resolvent-obstruction} yields
\(p\leq \frac{2(d+2)}{d+5}\).
Thus $p\leq \frac{2(d+2)}{d+5}$ is necessary when $d$ is odd. 

\subsection{Away from thresholds in even dimensions}
Let $d\ge4$ be even. Set $\lambda=2d-1$ and $\xi^0=(1/4,\dots,1/4,1/6)$. Define
\[\widehat f_\varepsilon(\xi) =
\prod_{j=1}^{d-2} \psi\!\left( \frac{\xi_j-\frac14}{\varepsilon^{1/3}} \right)
\psi\!\left( \frac{ \sum_{j=1}^{d-1}(\xi_j-\frac14) }{\varepsilon^{1/2}} \right)
\psi\!\left( \frac{ \sum_{j=1}^{d-1}(\xi_j-\frac14) + \frac{\sqrt3}{2}(\xi_d-\frac16) }{\varepsilon} \right).\]
The associated linear map is
\[L(\xi-\xi^0) = \left(\xi_1-\frac14,\dots,\xi_{d-2}-\frac14,\sum_{j=1}^{d-1}(\xi_j-\frac14),\ \sum_{j=1}^{d-1}(\xi_j-\frac14)+\frac{\sqrt3}{2}(\xi_d-\frac16)\right),\]
so $|\det L|=\sqrt3/2$. The scales are
$\rho_1=\cdots=\rho_{d-2}=\varepsilon^{1/3}$,
$\rho_{d-1}=\varepsilon^{1/2}$, $\rho_d=\varepsilon$, hence
\(V_\varepsilon = (\varepsilon^{1/3})^{d-2}\varepsilon^{1/2}\varepsilon = \varepsilon^{(2d+5)/6}\).
On the support we have $|\xi_j-1/4|\leq\varepsilon^{1/3}$ for $j=1,\dots,d-2$,
$|\sum_{j=1}^{d-1}(\xi_j-1/4)|\leq\varepsilon^{1/2}$, and
$|\sum_{j=1}^{d-1}(\xi_j-1/4)+\frac{\sqrt3}{2}(\xi_d-1/6)|\leq\varepsilon$. Thus
\[|\xi_{d-1}-1/4| \leq \left|\sum_{j=1}^{d-1}(\xi_j-1/4)\right| + \sum_{j=1}^{d-2}|\xi_j-1/4| \leq C_d\varepsilon^{1/3},\]
and
\[|\xi_d-1/6| \leq \frac{2}{\sqrt3}\left|\sum_{j=1}^{d-1}(\xi_j-1/4)+\frac{\sqrt3}{2}(\xi_d-1/6)\right|+ \frac{2}{\sqrt3}\left|\sum_{j=1}^{d-1}(\xi_j-1/4)\right|\leq C\varepsilon^{1/2}.\]
For $j=1,\dots,d-1$, $\cos(2\pi\xi_j)=-2\pi(\xi_j-1/4)+O(|\xi_j-1/4|^3)$. Also
\[\cos(2\pi\xi_d)=\frac12-\sqrt3\pi(\xi_d-1/6)-\pi^2(\xi_d-1/6)^2+O(|\xi_d-1/6|^3).\]
Consequently,
\begin{align*}
h_0(\xi)-(2d-1)
={}&4\pi\left[
\sum_{j=1}^{d-1}(\xi_j-1/4)
+\frac{\sqrt3}{2}(\xi_d-1/6)
\right] \\
&+2\pi^2(\xi_d-1/6)^2
+O\!\left(
\sum_{j=1}^{d-1}|\xi_j-1/4|^3+|\xi_d-1/6|^3
\right).
\end{align*}
On the support, the linear term is bounded by $C\varepsilon$, the quadratic term by $C\varepsilon$, and the cubic terms by $C\varepsilon$ (since $|\xi_d-1/6|^3\leq C\varepsilon^{3/2}\leq C\varepsilon$). Hence
$|h_0(\xi)-(2d-1)|\leq C_d\varepsilon$. By Lemma~\ref{lem:anisotropic-box},
\(\|f_\varepsilon\|_{\ell^p} \leq C_p \varepsilon^{\frac{2d+5}{6}(1-1/p)}\).
Since \(d\) is even, \(2d\in\{0,4,\dots,4d\}\), while
\(\lambda=2d-1\) satisfies
\[
\operatorname{dist}
\bigl(2d-1,\{0,4,\dots,4d\}\bigr)
=
1.
\]
If, for some fixed \(0<\delta<1\),
\[
\sup_{z\in D_\delta\setminus\mathbb R}
\|(H_0-z)^{-1}\|_{\ell^p\to\ell^{p'}}
<\infty,
\]
then we take \(z_\varepsilon=2d-1+i\varepsilon\) and apply
Lemma~\ref{lem:resolvent-obstruction} to obtain
\(p \leq \frac{ 2\frac{2d+5}{6} }{ \frac{2d+5}{6}+1 } = \frac{2(2d+5)}{2d+11}\).
Thus $p\leq \frac{2(2d+5)}{2d+11}$ is necessary when $d$ is even.

\begin{proof}[Proof of Theorem~\ref{thm:sharpness}]
The global construction at \(\lambda=2d\) gives
\(p\leq2(d+2)/(d+5)\).  If \(d\) is odd, this energy is a distance \(2\)
from every threshold, so the same obstruction applies away from the
thresholds.  If \(d\) is even, the construction at \(\lambda=2d-1\)
gives \(p\leq2(2d+5)/(2d+11)\).  These are precisely the assertions of
Theorem~\ref{thm:sharpness}.
\end{proof}

\subsection{Further fixed-energy obstructions}

All functions below are first defined on the fundamental domain
\((-1/2,1/2]^d\) and then extended periodically.  For sufficiently
small \(\varepsilon\), their supports lie in the interior of the
fundamental domain.

We use Lemmas~\ref{lem:anisotropic-box} and
\ref{lem:resolvent-obstruction} throughout.  In each construction below,
write \(y=\xi-\xi^0\).  The \(d\) linear forms appearing in the
arguments of \(\psi\) have determinant of absolute value \(1\), and a
Taylor expansion at \(\xi^0\) gives
\[
|h_0(\xi)-\lambda|\leq C_{\lambda,d}\varepsilon
\]
on the corresponding inner box.  At \(\lambda=4\), the same conclusion
follows from the addition formula for
\(\cos(2\pi\xi_1)+\cos(2\pi\xi_2)\).  Thus, if
\(V_\varepsilon\simeq\varepsilon^a\), a uniform resolvent estimate
forces
\[
p\leq\frac{2a}{a+1}.
\]
It suffices to consider \(0<\lambda\leq2d\), since
\[
h_0\left(
\xi_1+\frac12,\dots,\xi_d+\frac12
\right)
=
4d-h_0(\xi).
\]
The reflected energies \(4d-\lambda\) are treated by translating the
same constructions.

\begin{remark}\label{rem:fixed-energy-obstructions}
The resulting exponents are
\[
\begin{array}{c|c|c}
\text{energy} & a & \text{necessary bound} \\ \hline
0<\lambda<4
& \dfrac{d+1}{2}
& p\leq\dfrac{2(d+1)}{d+3} \\[4pt]
\lambda=4
& \dfrac d2
& p\leq\dfrac{2d}{d+2} \\[4pt]
4<\lambda<6
& \dfrac{2d+1}{4}
& p\leq\dfrac{2(2d+1)}{2d+5} \\[4pt]
2m\leq\lambda<2(m+1)
& \dfrac{3d-m+4}{6}
& p\leq\dfrac{2(3d-m+4)}{3d-m+10} \\[4pt]
\lambda=2d
& \dfrac{d+2}{3}
& p\leq\dfrac{2(d+2)}{d+5}
\end{array}
\]
where \(3\leq m\leq d-1\) in the fourth row.
\end{remark}

\medskip
\noindent\emph{Energies \(0<\lambda<4\).}
Set
\[
\alpha
=
\frac{1}{2\pi}
\arccos\left(1-\frac{\lambda}{2d}\right),
\qquad
\xi^0=(\alpha,\dots,\alpha),
\]
and define
\[
\widehat f_\varepsilon(\xi)
=
\prod_{j=1}^{d-1}
\psi\left(\frac{y_j}{\varepsilon^{1/2}}\right)
\psi\left(\frac{\sum_{j=1}^d y_j}{\varepsilon}\right).
\]
Then
\[
V_\varepsilon
=
(\varepsilon^{1/2})^{d-1}\varepsilon
=
\varepsilon^{(d+1)/2}.
\]

\medskip
\noindent\emph{Energy \(\lambda=4\).}
Set
\[
\xi^0=\left(\frac14,\frac14,0,\dots,0\right)
\]
and fix a sufficiently small constant \(0<\delta<1/100\).  Define
\[
\widehat f_\varepsilon(\xi)
=
\psi\left(\frac{y_1}{\delta}\right)
\prod_{j=3}^d
\psi\left(\frac{y_j}{\varepsilon^{1/2}}\right)
\psi\left(\frac{y_1+y_2}{\varepsilon}\right).
\]
Then
\[
V_\varepsilon
=
\delta(\varepsilon^{1/2})^{d-2}\varepsilon
\simeq
\varepsilon^{d/2}.
\]

\medskip
\noindent\emph{Energies \(4<\lambda<6\).}
Set
\[
\alpha
=
\frac{1}{2\pi}
\arccos\left(\frac{d-\lambda/2}{d-2}\right),
\qquad
\xi^0
=
\left(
\frac14,\frac14,\alpha,\dots,\alpha
\right),
\]
and define
\[
\widehat f_\varepsilon(\xi)
=
\psi\left(\frac{y_1}{\varepsilon^{1/4}}\right)
\prod_{j=3}^d
\psi\left(\frac{y_j}{\varepsilon^{1/2}}\right)
\psi\left(
\frac{
y_1+y_2+\sin(2\pi\alpha)\sum_{j=3}^d y_j
}{\varepsilon}
\right).
\]
Then
\[
V_\varepsilon
=
\varepsilon^{1/4}
(\varepsilon^{1/2})^{d-2}
\varepsilon
=
\varepsilon^{(2d+1)/4}.
\]

To verify the localization in this case, set
\[
S=\sum_{j=3}^d y_j,
\qquad
u=y_1+y_2+\sin(2\pi\alpha)S.
\]
On the inner box,
\[
|u|\lesssim\varepsilon,
\qquad
|S|\lesssim\varepsilon^{1/2},
\qquad
|y_1|+|y_2|\lesssim\varepsilon^{1/4},
\qquad
|y_1+y_2|\lesssim\varepsilon^{1/2}.
\]
Hence
\[
y_1^3+y_2^3
=
(y_1+y_2)(y_1^2-y_1y_2+y_2^2)
=
O(\varepsilon).
\]
Expanding \(h_0\) at \(\xi^0\), the linear part is \(4\pi u\), the
quadratic terms in \(y_3,\dots,y_d\) are \(O(\varepsilon)\), the cubic
terms in the two flat variables have the cancellation above, and all
remaining terms are \(O(\varepsilon)\).  Therefore
\[
|h_0(\xi)-\lambda|
\leq
C_{\lambda,d}\varepsilon
\]
on the inner box.

\medskip
\noindent\emph{Higher intermediate energies.}
Let
\[
3\leq m\leq d-1,
\qquad
2m\leq\lambda<2(m+1),
\]
and set
\[
\alpha
=
\frac{1}{2\pi}
\arccos\left(\frac{d-\lambda/2}{d-m}\right),
\qquad
\xi^0
=
\left(
\underbrace{\frac14,\dots,\frac14}_{m},
\underbrace{\alpha,\dots,\alpha}_{d-m}
\right).
\]
Define
\[
\widehat f_\varepsilon(\xi)
=
\prod_{j=1}^{m-1}
\psi\left(\frac{y_j}{\varepsilon^{1/3}}\right)
\prod_{j=m+1}^d
\psi\left(\frac{y_j}{\varepsilon^{1/2}}\right)
\psi\left(
\frac{
\sum_{j=1}^m y_j
+
\sin(2\pi\alpha)\sum_{j=m+1}^d y_j
}{\varepsilon}
\right).
\]
Then
\[
V_\varepsilon
=
(\varepsilon^{1/3})^{m-1}
(\varepsilon^{1/2})^{d-m}
\varepsilon
=
\varepsilon^{(3d-m+4)/6}.
\]

\medskip
\noindent\emph{The central energy \(\lambda=2d\).}
Set
\[
\xi^0=\left(\frac14,\dots,\frac14\right)
\]
and define
\[
\widehat f_\varepsilon(\xi)
=
\prod_{j=1}^{d-1}
\psi\left(\frac{y_j}{\varepsilon^{1/3}}\right)
\psi\left(\frac{\sum_{j=1}^d y_j}{\varepsilon}\right).
\]
Then
\[
V_\varepsilon
=
(\varepsilon^{1/3})^{d-1}\varepsilon
=
\varepsilon^{(d+2)/3}.
\]

\section*{Acknowledgment}
The author thanks Dr.~Adam Black for valuable discussions and for a
careful reading of earlier versions of the manuscript.  His
comments and suggestions helped clarify several points and correct errors
in an earlier draft.
\bibliographystyle{plain}
\bibliography{references}

@article{Cuenin2019,
	title = {Eigenvalue {Estimates} for {Bilayer} {Graphene}},
	volume = {20},
	issn = {1424-0637, 1424-0661},
	url = {http://link.springer.com/10.1007/s00023-019-00770-x},
	doi = {10.1007/s00023-019-00770-x},
	language = {en},
	number = {5},
	urldate = {2026-08-09},
	journal = {Annales Henri Poincaré},
	author = {Cuenin, Jean-Claude},
	month = may,
	year = {2019},
	pages = {1501--1516},
}

@article{KeelTao1998,
	title = {Endpoint {Strichartz} estimates},
	volume = {120},
	issn = {1080-6377},
	url = {https://muse.jhu.edu/article/811},
	doi = {10.1353/ajm.1998.0039},
	language = {en},
	number = {5},
	urldate = {2026-08-09},
	journal = {American Journal of Mathematics},
	author = {Keel, Markus and Tao, Terence},
	month = oct,
	year = {1998},
	pages = {955--980},
}

@article{Taira2021,
	title = {Uniform resolvent estimates for the discrete {Schrödinger} operator in dimension three},
	volume = {11},
	issn = {1664-039X, 1664-0403},
	url = {https://ems.press/doi/10.4171/jst/387},
	doi = {10.4171/jst/387},
	number = {4},
	urldate = {2026-08-09},
	journal = {Journal of Spectral Theory},
	author = {Taira, Kouichi},
	month = dec,
	year = {2021},
	pages = {1831--1855},
}

@article{KenigRuizSogge1987,
  author  = {Kenig, Carlos E. and Ruiz, Alberto and Sogge, Christopher D.},
  title   = {Uniform {S}obolev inequalities and unique continuation for second order constant coefficient differential operators},
  journal = {Duke Math. J.},
  volume  = {55},
  number  = {2},
  year    = {1987},
  pages   = {329--347},
  doi     = {10.1215/S0012-7094-87-05518-9}
}

@article{ErdosSalmhofer2007,
  author  = {Erd{\H{o}}s, L{\'a}szl{\'o} and Salmhofer, Manfred},
  title   = {Decay of the {F}ourier transform of surfaces with vanishing curvature},
  journal = {Math. Z.},
  volume  = {257},
  number  = {2},
  year    = {2007},
  pages   = {261--294},
  doi     = {10.1007/s00209-007-0125-4}
}

@article{KorotyaevMoller2019,
  author  = {Korotyaev, Evgeny L. and M{\o}ller, Jacob Schach},
  title   = {Weighted estimates for the {L}aplacian on the cubic lattice},
  journal = {Ark. Mat.},
  volume  = {57},
  number  = {2},
  year    = {2019},
  pages   = {397--428},
  doi     = {10.4310/ARKIV.2019.v57.n2.a8}
}

@article{TadanoTaira2019,
  author  = {Tadano, Yukihide and Taira, Kouichi},
  title   = {Uniform bounds of discrete {B}irman--{S}chwinger operators},
  journal = {Trans. Amer. Math. Soc.},
  volume  = {372},
  number  = {7},
  year    = {2019},
  pages   = {5243--5262},
  doi     = {10.1090/tran/7882}
}

@book{Stein1993,
  title={Harmonic Analysis: Real-Variable Methods, Orthogonality, and Oscillatory Integrals},
  author={Stein, Elias M.},
  year={1993},
  publisher={Princeton University Press}
}

@misc{BlackDroginHernandez2025,
	title = {Self-consistent equations and quantum diffusion for the {Anderson} model},
	url = {http://arxiv.org/abs/2506.06468},
	doi = {10.48550/arXiv.2506.06468},
	urldate = {2026-08-26},
	publisher = {arXiv},
	author = {Black, Adam and Drogin, Reuben and Hernández, Felipe},
	month = nov,
	year = {2025},
	note = {arXiv:2506.06468 [math-ph]},
}

@article{Frank2011,
	title = {Eigenvalue bounds for {Schrödinger} operators with complex potentials},
	volume = {43},
	copyright = {http://doi.wiley.com/10.1002/tdm\_license\_1.1},
	issn = {00246093},
	url = {http://doi.wiley.com/10.1112/blms/bdr008},
	doi = {10.1112/blms/bdr008},
	language = {en},
	number = {4},
	urldate = {2026-08-26},
	journal = {Bulletin of the London Mathematical Society},
	author = {Frank, Rupert L.},
	month = aug,
	year = {2011},
	pages = {745--750},
}

@article{KatoYajima1989,
	title = {Some examples of smooth operators and the associated smoothing effect},
	volume = {01},
	issn = {0129-055X, 1793-6659},
	url = {https://www.worldscientific.com/doi/abs/10.1142/S0129055X89000171},
	doi = {10.1142/S0129055X89000171},
	language = {en},
	number = {04},
	urldate = {2026-08-26},
	journal = {Reviews in Mathematical Physics},
	author = {Kato, Tosio and Yajima, Kenji},
	month = jan,
	year = {1989},
	pages = {481--496},
}
\end{document}